\documentclass[12pt,A4,leqno]{amsart}
\usepackage{amsfonts}
\usepackage{mathrsfs}
\usepackage[T1]{fontenc}
\usepackage{amssymb,amscd}
\usepackage{upgreek}
\usepackage{mathrsfs} 
\usepackage{color}
\usepackage{epsf}
\usepackage{graphicx}
\usepackage{extarrows}
\usepackage[curve]{xypic}
\usepackage{hyperref} 
\usepackage{url}
\hypersetup{colorlinks,linkcolor={blue},citecolor={blue},urlcolor={blue}} 
\usepackage[capitalise]{cleveref}

\theoremstyle{plain}
\newtheorem{thm}{Theorem}[section]
\newtheorem{lem}[thm]{Lemma}
\newtheorem{prop}[thm]{Proposition}
\newtheorem{cor}[thm]{Corollary}

\theoremstyle{definition}
\newtheorem{defi}[thm]{Definition}
\newtheorem{exam}[thm]{Example}
\newtheorem{rem}[thm]{Remark}

\newcommand{\R}{\mathbb R}
\newcommand{\Z}{\mathbb Z}

\newcommand{\nn}{\vskip 0.2cm}
\newcommand{\n}{\vskip 0.1cm}

\makeatletter
\renewcommand{\@secnumfont}{\bfseries}
\makeatother

\makeatletter
\def\section{%
  \@startsection{section}{1}
    {\z@}
    {2.0ex plus 0.8ex minus .1ex}
    {1.0ex plus .2ex}
    {\bfseries\large\centering\MakeUppercase}%
}
\makeatother

\begin{document}

\title [\ ] {A unifying view toward polyhedral products through panel structures}

\author{Li Yu}
\address{School of Mathematics, Nanjing University, Nanjing, 210093, P.R.China.
  }
 \email{yuli@nju.edu.cn}


\keywords{Polyhedral product, moment-angle complex, panel structure, torus action, topological face ring}


\thanks{2020 \textit{Mathematics Subject Classification}. 57S12 (Primary), 57S25, 55N91, 13F55.
 }

\begin{abstract}
A panel structure on a topological space is just a locally finite family of closed subspaces. A space together
with a panel structure is called a space with faces.
In this paper, we introduce a notion of polyhedral product over a space with faces. This notion provides a unifying viewpoint on the constructions of polyhedral products and generalized moment-angle complexes in various settings.
 We compute the stable decomposition of these spaces and use it to study their cohomology ring structures. Moreover, we can compute the equivariant cohomology ring of the moment-angle complex over a space with faces with respect to the canonical torus action. The calculation leads to the notion of topological face ring of a space with faces, which generalizes the classical notion of face ring (Stanley-Reisner ring) of a simplicial complex. We will see that many known results in the study of polyhedral products and moment-angle complexes can be reinterpreted from our general theorems on the polyhedral product over a space with faces. Moreover, we can derive some new results via our approach in some settings. 
  \end{abstract}

\maketitle

 \section{Introduction}
 
 \subsection{Moment-angle complex and polyhedral product}\ \n
 
  The construction of moment-angle complex $\mathcal{Z}_K$ over a simplicial complex $K$ first appeared in
  Davis and Januszkiewicz's pioneering work~\cite{DaJan91}. The name ``moment-angle complex'' was given by Buchstaber and Panov~\cite{BP00} afterwards. The more general definition of polyhedral product determined by a simplicial complex and a sequence of space pairs was introduced by Bahri, Bendersky, Cohen and Gitler~\cite{BBCG10}, which has become the major subject in the homotopy theoretic study in toric topology. The reader is also referred to~\cite{BBCG10} for more information on the earlier 
 works related to polyhedral products by 
 different authors. \n

  Let $K$ denote an (abstract) \emph{simplicial complex} with $m$ vertices labeled by the set $[m]=\{1,2,\cdots,m\}$. 
   We think of each $(k-1)$-simplex $\sigma$ of $K$ as
   a subset $\{i_1,\cdots, i_k\}\subseteq [m]$ such that if $\tau\subseteq \sigma$, then $\tau$ is a simplex of $K$. 
   In particular, the empty set $\varnothing$ is a simplex of $K$.\n
      
   Let $(\mathbb{X},\mathbb{A}) = \{ (X_j,A_j,a_j) \}^m_{j=1}$ be a sequence of based CW-complexes where $a_j\in A_j\subseteq X_j$, $1\leq j \leq m$.
    The \emph{polyhedral product} (or \emph{generalized moment-angle complex}) determined by $(\mathbb{X},\mathbb{A})$ and $K$, denoted by
    $(\mathbb{X},\mathbb{A})^K$ or $\mathcal{Z}(K;(\mathbb{X},\mathbb{A}))$, is defined to be 
     \begin{equation} \label{Equ:Poly-Prod-Simplicial}
       (\mathbb{X},\mathbb{A})^K =  \bigcup_{\sigma\in K} (\mathbb{X},\mathbb{A})^{\sigma} \, \subseteq \prod_{j\in [m]} X_j,\ \text{where}\ 
  (\mathbb{X},\mathbb{A})^{\sigma} = \prod_{j\in\sigma} X_j \times \prod_{j\in [m]\backslash \sigma} A_j.
    \end{equation}
    
    If $(\mathbb{X},\mathbb{A}) = \{ (X_j,A_j,a_j)=(X,A,a_0) \}^m_{j=1}$, we also denote 
    $(\mathbb{X},\mathbb{A})^K$ by
   $(X,A)^K$. In particular, 
   $(D^2,S^1)^K$ and $(D^1,S^0)^K$ are called the \emph{moment-angle complex} and \emph{real moment-angle complex} over $K$, respectively.\n
     
  We can extend the construction of $(\mathbb{X},\mathbb{A})^K$
   to
  polyhedral product over any simplicial poset (or simplicial cell complex) via the colimit construction. Such a generalization has been done by L\"u and  Panov~\cite{LuPanov11} for moment-angle complexes.\n 
  
  Another generalization is the \emph{moment-angle 
  space} of a convex polytope (not necessarily simple) introduced by Ayzenberg and Buchstaber~\cite{AntonBuch11}.\n
         
  More recently, 
   a notion of polyhedral product over a nice manifold with corners was introduced by Yu~\cite{Yu20}. The main idea of the construction in~\cite{Yu20} suggests that we can define
 a more general notion of polyhedral product 
  which unifies the constructions of various
 polyhedral products (up to homotopy equivalences). This is main purpose of this paper.
 \vskip .2cm
 
 \subsection{Cohomology ring} \ \n

  The cohomology ring of the moment-angle
complex over a simplicial complex $K$ was computed by Franz~\cite{Franz06} and  Baskakov, Buchstaber and Panov~\cite{BasBuchPanov04}. 
It is shown in~\cite{BasBuchPanov04} that the cohomology 
ring of $\mathcal{Z}_K$ with coefficients $\mathbf{k}$ (a commutative ring with unit) is isomorphic to the Tor-algebra
of the \emph{face ring} (or \emph{Stanley-Reisner ring}) $\mathbf{k}[K]$ of $K$ using the Koszul resolution of $\mathbf{k}[K]$.
 Recall $\mathbf{k}[K]$ is the quotient 
      \begin{equation} \label{Equ:Face-Ring-Def-1}
         \mathbf{k}[K]=\mathbf{k}[v_1,\cdots, v_m]\slash \mathcal{I}_{K}
       \end{equation}  
      where $\mathcal{I}_{K}$ 
      is the ideal generated by the monomials $v_{i_1}\cdots v_{i_s}$ for which 
      $\{i_1,\cdots, i_s\}$ does not span a
       simplex of $K$. A linear basis of $\mathbf{k}[K]$
       over $\mathbf{k}$ is given by
      \begin{equation} \label{Equ:Face-Ring-Basis}
       \{1\}\cup \big\{  v^{n_1}_{i_1}\cdots v^{n_s}_{i_s}\,|\,
        \varnothing\neq \{i_1,\cdots,i_s\} \in K,   \,  n_1>0,\cdots, n_s>0  \big\}. 
      \end{equation}
    \n

The cohomology 
ring of $\mathcal{Z}_K$ was also computed by Wang and Zheng~\cite{WangZheng15} using the simplicial complement of $K$ and
the Taylor resolution of $\mathbf{k}[K]$.\n

 For a general polyhedral product $(\mathbb{X},\mathbb{A})^K$, its cohomology ring structure was computed by Bahri, Bendersky, Cohen and Gitler~\cite{BBCG12} via partial diagonal maps and by Bahri, Bendersky, Cohen and Gitler~\cite{BBCG17}
by a spectral sequence 
under certain freeness conditions (coefficients in a field for example).
The study in this direction is further extended by Bahri, Bendersky, Cohen and Gitler~\cite{BBCG20}.
A computation using a different method was carried out by Zheng~\cite{Zheng16}.
 In this paper, we will mainly use the 
strategy in~\cite{BBCG12} to study the cohomology ring
of our polyhedral products.
\vskip .2cm

\subsection{Equivariant cohomology ring}\ \n
 
   For a topological group $G$, denote by $G\rightarrow EG\rightarrow BG$ the universal  principal $G$-bundle, with $EG$ a contractible space endowed with a free $G$-action, and $BG$ the orbit space. We may view $G$ as a subspace of $EG$ (the orbit of a basepoint).\n
   
    The \emph{Borel construction} of a space $X$ with a $G$-action is the quotient space
     $$  EG \times_G X = 
     EG \times X \big\slash \sim $$
    where $(e,x)\sim (eg,g^{-1}x)$ for any $e\in 
    EG$, $x\in X$ and $g\in G$.\n  
    
    Associated to the Borel construction, there is a canonical fiber bundle
    $$ X \rightarrow EG \times_G X \rightarrow BG. $$ 
    
    The equivariant cohomology ring of $X$, denoted by $H^*_G(X)$ is the (ordinary) cohomology ring of the Borel construction of $X$.\n
    
   In addition, if a subspace $A$ of $X$ is invariant under the $G$-action on $X$, we call $(X,A)$ a \emph{$G$-space pair}. Then     
   if in $(\mathbb{X},\mathbb{A}) = \{ (X_j,A_j,a_j) \}^m_{j=1}$, each $(X_j,A_j)$ is a $G_j$-space pair,
   the $G_j$-actions on the pairs $(X_j,A_j)$ extend canonically
   to an action of the product group $\prod_{j\in [m]} G_j$
   on the pair $\big(\prod_{j\in [m]} X_j, \prod_{j\in [m]} A_j\big)$. It is easy
   to check that this action of $\prod_{j\in [m]} G_j$ preserves the
   subspace $ (\mathbb{X},\mathbb{A})^K \subseteq \prod_{j\in [m]} X_j$ for any abstract simplicial complex $K$ on $[m]$. Then it is meaningful to study
   the equivariant cohomology ring of
   $(\mathbb{X},\mathbb{A})^K$ with respect to the $\prod_{j\in [m]} G_j$-action. \n
   
   When $G_1=\cdots=G_m = S^1$, the equivariant cohomology
   ring of $(D^2,S^1)^K$ with respect to the canonical
   $(S^1)^m$-action is shown to be isomorphic to
   the face ring $\Z[K]$ of $K$ (see~\cite{DaJan91} and Buchstaber and Panov~\cite{BP15}).
   This result was generalized to moment-angle complexes over simplicial posets in~\cite{LuPanov11}.
   \vskip .2cm
   
 \subsection{Panel structure}\ \n
 
  A \emph{panel structure} on a topological space $Y$ is just a locally finite family of closed subspaces
$\mathcal{P}=\{ P_{j} \}_{j\in \Lambda}$  (see Davis~\cite[Section 6]{Da83}). Each subspace $P_j$ is called a \emph{panel}. The space $Y$
 together with a panel structure $\mathcal{P}$ is called a \emph{space with faces}, denoted by $(Y,\mathcal{P})$.\n

   Suppose $\Lambda = [m]=\{1,\cdots, m\}$ is a finite set. For any subset $J\subseteq [m]$, let
    $$ P_J = \bigcup_{j\in J} P_j,\ P_{\varnothing}=\varnothing; \ \  P_{\cap J} = \bigcap_{j\in J} P_j,\
     P_{\cap\varnothing} =Y. $$

 It is clear that
 \[    P_{\cap J} \subseteq P_{J},\ \ P_J\subseteq P_{J'}, \ \ P_{\cap J'} \subseteq P_{\cap J},\ 
  \forall J\subseteq J' \subseteq [m]. \]
 
   Let
 $$Y_{c(\mathcal{P})} = Y\backslash P_{[m]}= Y\backslash  (P_1\cup \cdots \cup P_m)$$
  We called $Y_{c(\mathcal{P})}$ the \emph{core} of 
  $(Y,\mathcal{P})$.\n  
  
   Each path-connected component of $P_{\cap J}$
   is called a \emph{face} of $(Y,\mathcal{P})$.
   In particular, each path-connected component of $Y=P_{\cap\varnothing}$ is a face of $(Y,\mathcal{P})$.\n   
  In addition, let $\mathcal{S}_{(Y,\mathcal{P})}$ denote the \emph{face poset} of $(Y,\mathcal{P})$, which is 
    the set of all faces of $(Y,\mathcal{P})$ ordered by inclusion.
    \vskip .2cm
        
    \subsection{Polyhedral product over a space with faces}
    \ \n
     
    Let $\mathcal{P}=\{ P_j \}_{j\in [m]}$ be a panel structure on a topological space $Y$. For any face $f$ of $Y$,  the following subset of $[m]$ is called the \emph{panel index} of $f$.
  \begin{equation} \label{Equ:Face-Index}
    I_f = \{ j\in [m] \,|\, f\subseteq P_j \} \subseteq [m].
  \end{equation}

   Let $(\mathbb{X},\mathbb{A}) = \{ (X_j,A_j,a_j) \}^m_{j=1}$ be a sequence of based CW-complexes where $a_j\in A_j\subseteq X_j$, $1\leq j \leq m$.
      The \emph{polyhedral product of $(\mathbb{X},\mathbb{A})$ over $(Y,\mathcal{P})$} is
      $$ (\mathbb{X},\mathbb{A})^{(Y,\mathcal{P})} = \bigcup_{f\in \mathcal{S}_{(Y,\mathcal{P})}} (\mathbb{X},\mathbb{A})^f\, \subseteq
      Y\times \prod_{j\in [m]} X_j, \ \text{where} $$
    $$  (\mathbb{X},\mathbb{A})^f = 
     f \times \prod_{j\in I_f} X_j \times \prod_{j\in [m]\backslash I_f} A_j, \ \forall f\in \mathcal{S}_{(Y,\mathcal{P})}.$$    
    If $(\mathbb{X},\mathbb{A}) = \{ (X_j,A_j,a_j)=(X,A,a_0) \}^m_{j=1}$, we also denote 
    $(\mathbb{X},\mathbb{A})^{(Y,\mathcal{P})}$ by
   $(X,A)^{(Y,\mathcal{P})}$. In particular, we call
   $(D^2,S^1)^{(Y,\mathcal{P})}$ and $(D^1,S^0)^{(Y,\mathcal{P})}$ the \emph{moment-angle complex} and \emph{real moment-angle complex} over $(Y,\mathcal{P})$, respectively.\n

   \begin{exam}
      Let $\mathcal{P}=\{ P_j \}_{j\in [m]}$ be a panel structure on $Y$.
      \begin{itemize}
      \item If $P_1=\cdots=P_m=Y$, then $(\mathbb{X},\mathbb{A})^{(Y,\mathcal{P})} = Y\times  \prod_{j\in [m]} X_j$.\n
       
      \item If $P_1=\cdots=P_m=\varnothing$, then $(\mathbb{X},\mathbb{A})^{(Y,\mathcal{P})} = Y\times \prod_{j\in [m]} A_j$.\n
       
      \item  Let $\mathcal{P}=\{ P_j \}_{j\in [m]}$ and $\mathcal{P}'=\{ P'_j \}_{j\in [m]}$ be two panel structures on $Y$. If $P_j\subset P'_j$ for all $j\in [m]$, we call
      $\mathcal{P}$ a \emph{refinement} of $\mathcal{P}'$ and write
      $\mathcal{P} \preccurlyeq \mathcal{P}'$.
      It is clear that if $\mathcal{P} \preccurlyeq \mathcal{P}'$, then
   $(\mathbb{X},\mathbb{A})^{(Y,\mathcal{P})} \subseteq
     (\mathbb{X},\mathbb{A})^{(Y,\mathcal{P}')}$.        
      \end{itemize}

   So we obtain a hierarchy of spaces between $Y \times \prod_{j\in [m]} A_j$ and
   $Y\times  \prod_{j\in [m]} X_j$ given by the polyhedral products of $(\mathbb{X},\mathbb{A})$ over all the panel structures on $Y$.\n
   
 In categorical languages, let $\mathscr{L}_m(Y)$ be the category of all the panel structures on $Y$ with $m$ panels where the morphisms
 between two panel structures $\mathcal{P} \preccurlyeq \mathcal{P}'$ are the inclusions of corresponding panels.
 Then every $(\mathbb{X},\mathbb{A}) = \{ (X_j,A_j,a_j) \}^m_{j=1}$ determines a functor from $\mathscr{L}_m(Y)$ to
 the category of topological spaces which maps any
  $(Y,\mathcal{P})\in \mathscr{L}_m(Y)$ to the polyhedral product $(\mathbb{X},\mathbb{A})^{(Y,\mathcal{P})}$.
   \end{exam}
   
  Similarly to the study of moment-angle complexes and polyhedral products over simplicial complexes, we consider the following problems in this paper.
   \begin{itemize}
    \item Compute the stable decomposition of $(\mathbb{X},\mathbb{A})^{(Y,\mathcal{P})}$.\n
    \item Describe the cohomology ring structure of
     $(\mathbb{X},\mathbb{A})^{(Y,\mathcal{P})}$.\n
    \item Compute the equivariant cohomology of 
     $(D^2,S^1)^{(Y,\mathcal{P})}$ and  $(D^1,S^0)^{(Y,\mathcal{P})}$. 
    \end{itemize}

 \n
 
 We will study these problems in Section \ref{Sec:Stable-Decomp}, \ref{Sec:Cohom-Ring} and \ref{Sec:Equiv-Cohom}, respectively.
 The philosophy in our study is to treat a
 CW-complex with a panel structure as a polyhedron where
 each panel plays the role of a facet (codimension-one face). 
Therefore, many theorems and their proofs in Section \ref{Sec:Stable-Decomp}, \ref{Sec:Cohom-Ring} and \ref{Sec:Equiv-Cohom} 
 are parallel to the theorems in~\cite[Section\,4,\ Section\,5]{Yu20}. Because of this, we will only write the statements of these theorems and omit their proofs.  \n
 
 In Section~\ref{Sec:Reexam-Simplical-Complex}, Section~\ref{Sec:Reexam-Simplicial-Posets} and Section~\ref{Sec:Reexam-Manifold-Corners},
 we will reinterpret the constructions of polyhedral products over simplicial complexes, simplicial posets and manifold with corners, respectively, from the viewpoint of panel structures on a space; and recover many known results on these spaces from our general theorems
 of polyhedral products over a space with faces.
 Moreover, we can derive some new results via our approach in these settings (e.g. Proposition~\ref{Prop:Stable-Decomp-RZ-K}, Proposition~\ref{Prop:Real-Moment-Angle},
 Proposition~\ref{Prop:Simp-Poset-Stable-Decomp} and Proposition~\ref{Prop:Tor-Explain}). Especially, we will give a new way to write the cohomology ring structure of the ordinary (real)
 moment-angle complex of a simplicial complex.
 \n
 
   \noindent \textbf{Convention:} In the rest of the paper, 
  we assume that $Y$ is a CW-complex and every panel of
  a panel structure $\mathcal{P}$ on $Y$ is a CW-subcomplex of $Y$. 
 In addition, we assume that
  the number of faces in $(Y,\mathcal{P})$ is finite, i.e.
  $\mathcal{S}_{(Y,\mathcal{P})}$ is a finite poset.  
  \n

 In this paper, the term ``cohomology'' of a space $X$, denoted by $H^*(X)$, always mean singular cohomology with integral coefficients if not specified otherwise.

 \vskip .5cm
    
      \section{Topology of $(\mathbb{X},\mathbb{A})^{(Y,\mathcal{P})}$}
   \subsection{Stable decomposition of $(\mathbb{X},\mathbb{A})^{(Y,\mathcal{P})}$} \label{Sec:Stable-Decomp}
   
   \ \n 
   
   To do the stable decomposition of 
   $(\mathbb{X},\mathbb{A})^{(Y,\mathcal{P})}$ as the stable decomposition of a polyhedral product over a simplicial complex in~\cite{BBCG10}, we want to first think of $(\mathbb{X},\mathbb{A})^{(Y,\mathcal{P})}$ as the colimit of 
a diagram of CW-complexes over a poset.
The following are some basic definitions (see~\cite{ZiegZiv93}).
   
   \begin{itemize}
    \item Let $CW$ be the category of CW-complexes and continuous maps.\n
    
    \item  Let $CW_*$ be the category of based CW-complexes and based continuous maps. \n
    
     \item A \emph{diagram} $\mathcal{D}$ of CW-complexes or based CW-complexes over a finite poset $\mathscr{P}$ is a functor 
   $$\mathcal{D} : \mathscr{P}\rightarrow CW \ 
   \text{or}\ CW_*$$
    so that for every $p\leq p'$ in $\mathscr{P}$, there is a map 
     $d_{pp'}: \mathcal{D}(p')\rightarrow \mathcal{D}(p)$ with
     $$d_{pp}=id_{\mathcal{D}(p)},\ \ d_{pp'} d_{p'p''} = d_{pp''},\ \forall\, p\leq p'\leq p''.$$    
   \n
   
   \item The \emph{colimit} of $\mathcal{D}$ is the space
     \[ \mathrm{colim}(\mathcal{D})= \underset{p\in\mathscr{P}}{\mathrm{colim}}\, \mathcal{D}(p):=\Big(\coprod_{p\in \mathscr{P}} \mathcal{D}(p)\Big) \Big\slash \sim \]
     where $\sim$ denotes the equivalence relation generated by requiring that for each $x\in \mathcal{D}(p')$,
     $x\sim d_{pp'}(x)$ for every $p < p'$. \nn 
    \end{itemize}
    
    Let $(Y,\mathcal{P})$ be a space with faces
      where $\mathcal{P}=\{ P_j \}_{j\in [m]}$ is a panel structure on $Y$.
    We can think of  $(\mathbb{X},\mathbb{A})^{(Y,\mathcal{P})}$
    as a colimit of CW-complexes as follows.
    Let 
  \begin{equation} \label{Equ:poset-Y-P}
   \mathscr{P}_{(Y,\mathcal{P})} := \{ (f , L)\,|\, f\in \mathcal{S}_{(Y,\mathcal{P})}, L\subseteq I_f  \}
   \end{equation}
 be a poset where the partial order $\leq$ on $\mathscr{P}_{(Y,\mathcal{P})}$ is defined by 
 \[
  \text{$(f, L)\leq (f', L')$ if and only if
    $f\supseteq f'$ and $I_f\backslash L \supseteq I_{f'}\backslash L'$. }\]
   
        Note that $\mathscr{P}_{(Y,\mathcal{P})}$ is a finite poset since by our convention $(Y,\mathcal{P})$ only has finitely many faces.\n
     
 For any $ (f, L) \in \mathscr{P}_{(Y,\mathcal{P})} $, let 
   \begin{equation}
     (\mathbb{X},\mathbb{A})^{(f,L)} = 
       f \times \prod_{j\in I_f\backslash L} X_j \times \prod_{j\in [m]\backslash (I_f\backslash L)} A_j. 
   \end{equation}
       
 Clearly, $ (\mathbb{X},\mathbb{A})^{(f,L)} \subseteq  (\mathbb{X},\mathbb{A})^{(f,\varnothing)} =
     (\mathbb{X},\mathbb{A})^f$. So 
     $$ (\mathbb{X},\mathbb{A})^{(Y,\mathcal{P})} = \bigcup_{f\in \mathcal{S}_{(Y,\mathcal{P})}} (\mathbb{X},\mathbb{A})^f = \bigcup_{(f, L)\in \mathscr{P}_{(Y,\mathcal{P})}} (\mathbb{X},\mathbb{A})^{(f,L)}.  $$
      
  Next, define a diagram of CW-complexes 
  $$\mathbf{D}_{(\mathbb{X},\mathbb{A})}:
  \mathscr{P}_{(Y,\mathcal{P})}\rightarrow CW,\ \, \mathbf{D}_{(\mathbb{X},\mathbb{A})}((f, L))= (\mathbb{X},\mathbb{A})^{(f, L)}$$ 
  where $\big(d_{(\mathbb{X},\mathbb{A})}\big)_{(f, L),(f',L')}: \mathbf{D}_{(\mathbb{X},\mathbb{A})}((f',L'))\rightarrow
  \mathbf{D}_{(\mathbb{X},\mathbb{A})}((f, L)) $ is the natural inclusion for any $(f, L) \leq (f',L')\in \mathscr{P}_{(Y,\mathcal{P})}$.\n
   
 It is easy to see that $(\mathbb{X},\mathbb{A})^{(Y,\mathcal{P})}$ is the colimit of $\mathbf{D}_{(\mathbb{X},\mathbb{A})}$, that is
 \begin{equation} \label{Equ:Colimit-Construc}
    (\mathbb{X},\mathbb{A})^{(Y,\mathcal{P})} = \mathrm{colim} \big( \mathbf{D}_{(\mathbb{X},\mathbb{A})} \big) . 
 \end{equation}

 Next, we compute the homotopy type of the \emph{suspension}
  $\mathbf{\Sigma}\big( (\mathbb{X},\mathbb{A})^{(Y,\mathcal{P})} \big)$ of $(\mathbb{X},\mathbb{A})^{(Y,\mathcal{P})}$ under some special conditions on $(\mathbb{X},\mathbb{A})$. 
 The following two theorems are generalizations of ~\cite[Theorem 4.4]{Yu20} and~\cite[Theorem 4.12]{Yu20}, respectively.  The ``$\bigvee$'' and ``$\bigwedge$'' are wedge sum and smash product of spaces, respectively.

 \begin{thm} \label{Thm:Main-General-X-Contract}
   Let $(\mathbb{X},\mathbb{A})= \{ (X_j,A_j,a_j) \}^m_{j=1}$ where each $X_j$ is contractible and each $A_j$ is either connected or is a disjoint union of a connected CW-complex with its basepoint. 
   For any panel structure $\mathcal{P}=\{P_j\}_{j\in [m]}$ on a CW-complex $Y$, there is a homotopy equivalence
      \begin{equation*}  
         \mathbf{\Sigma}\big(  (\mathbb{X},\mathbb{A})^{(Y,\mathcal{P})} \big) \simeq \bigvee_{J\subseteq [m]} 
           \mathbf{\Sigma} \Big( Y\slash P_J \wedge\bigwedge_{j\in J} A_j \Big).
           \end{equation*} 
So the reduced homology group 
  $$ \widetilde{H}_*\big(  (\mathbb{X},\mathbb{A})^{(Y,\mathcal{P})} \big) \cong \bigoplus_{J\subseteq [m]} \widetilde{H}_*\big(Y\slash P_J \wedge\bigwedge_{j\in J} A_j \big).$$    
    \end{thm}
    \begin{proof}
    The proof is completely parallel to the proof of~\cite[Theorem 4.4]{Yu20}.
    \end{proof}

   \begin{thm} \label{Thm:Stable-Decomp-A-Contract}
   Let $(\mathbb{X},\mathbb{A})= \{ (X_j,A_j,a_j) \}^m_{j=1}$ where each $A_j$ is contractible and each $X_j$ is either connected or is a disjoint union of a connected CW-complex with its basepoint. 
   For any panel structure $\mathcal{P}=\{P_j\}_{j\in [m]}$ on a CW-complex $Y$, there is a homotopy equivalence
  \begin{align*}
   S^1\vee \mathbf{\Sigma}\big( (\mathbb{X},\mathbb{A})^{(Y,\mathcal{P})} \big)  
 & \simeq \bigvee_{J\subseteq [m]} 
 \mathbf{\Sigma}  \Big(  \big( P_{\cap J}\cup y_0\big)  \wedge \bigwedge_{j\in J}  X_j \Big).
 \end{align*}
 where $y_0$ is a point disjoint from $Y$.
 This implies
 \[  H_*\big((\mathbb{X},\mathbb{A})^{(Y,\mathcal{P})} \big) \cong
   \bigoplus_{J\subseteq [m]} \widetilde{H}_* \Big(  \big( P_{\cap J}\cup y_0\big)  \wedge \bigwedge_{j\in J}  X_j \Big).\]
  \end{thm}
\begin{proof}
    The proof is parallel to the proof of~\cite[Theorem 4.12]{Yu20}.
    \end{proof}  

  \vskip .2cm

  \subsection{Cohomology ring of $(\mathbb{X},\mathbb{A})^{(Y,\mathcal{P})}$}  \label{Sec:Cohom-Ring}
  \ \n
 It was shown in~\cite{BBCG12} that the cohomology ring structure of a polyhedral product over a simplicial complex can be computed through
 the stable decomposition and partial diagonal maps of the polyhedral product.
For a panel structure $\mathcal{P}$ on a finite CW-complex $Y$,
since we also have the stable decomposition of $(\mathbb{X},\mathbb{A})^{(Y,\mathcal{P})}$, we can describe the cohomology ring of 
$(\mathbb{X},\mathbb{A})^{(Y,\mathcal{P})}$ in a similar way.\n

  Under some special conditions on $(\mathbb{X},\mathbb{A})$ (e.g. the conditions in Theorem~\ref{Thm:Main-General-X-Contract} and Theorem~\ref{Thm:Stable-Decomp-A-Contract}), we can write the cohomology ring structure of $(\mathbb{X},\mathbb{A})^{(Y,\mathcal{P})}$
  explicitly. For example, let 
    $$(\mathbb{D},\mathbb{S}) =
    \big\{ \big( D^{n_j+1}, S^{n_j}, a_j \big) \big\}^m_{j=1}$$
    where $D^{n+1}$ is the unit ball
      in $\R^{n+1}$ and $S^n=\partial D^{n+1}$.      
    We can describe the cohomology ring structure of
    $(\mathbb{D},\mathbb{S})^{(Y,\mathcal{P})}$ by the following ring. Let
  \begin{equation} \label{Equ:Cohomology-Ring-R}
    \mathcal{R}^*_{Y,\mathcal{P}} :=\bigoplus_{J \subseteq [m]} H^*(Y,P_J)
    \end{equation} 
  
  Define a graded ring structure $\Cup^{(\mathbb{D},\mathbb{S})}$ on $\mathcal{R}^*_{Y,\mathcal{P}} $ as follows. 
       \begin{itemize}
   \item If $J\cap J'=\varnothing$ or $J\cap J' \neq \varnothing$ but $n_j=0$ for all $j\in J\cap J'$,\\   
  $ H^*(Y,P_J) \otimes H^{*}(Y,P_{J'}) \xlongrightarrow{\Cup^{(\mathbb{D},\mathbb{S})}}
    H^{*}(Y,P_{J\cup J'})$ is the relative cup product $\cup$.\nn
          
     \item If $J\cap J' \neq \varnothing$ and there exists $n_j\geq 1$ for some $j\in J\cap J'$, 
    \\   
  $ H^*(Y,P_J) \otimes H^{*}(Y,P_{J'}) \xlongrightarrow{\Cup^{(\mathbb{D},\mathbb{S})}}
    H^{*}(Y,P_{J\cup J'})$ is trivial.
    \end{itemize}

   Let   \begin{equation*}
    \widetilde{\mathcal{R}}^*_{Y,\mathcal{P}} :=\bigoplus_{J \subseteq [m]} \widetilde{H}^*(Y\slash P_J), \ \text{then}\
    \mathcal{R}^*_{Y,\mathcal{P}} =  \widetilde{\mathcal{R}}^*_{Y,\mathcal{P}} \oplus \Z.
    \end{equation*} 
    
    By~\cite[Lemma 3.2]{Yu20}, there is a natural product on $\widetilde{\mathcal{R}}^*_{Y,\mathcal{P}}$,
    denoted by $\widetilde{\Cup}^{(\mathbb{D},\mathbb{S})}$,
    that is induced from the product $\Cup^{(\mathbb{D},\mathbb{S})}$ on $\mathcal{R}^*_{Y,\mathcal{P}}$.     
   So we have a commutative diagram 
     \begin{equation}  \label{Equ:Diagram-Induced-Prod}
   \xymatrix{
           H^*(Y,P_J) \otimes H^*(Y,P_{J'}) \ar[d] \ar[r]^{\quad\ \ \ \ \scalebox{0.85}{$\Cup^{(\mathbb{D},\mathbb{S})}$}}
                &   H^*(Y,P_{J\cup J'}) \ar[d]  \\
           \widetilde{H}^*(Y\slash P_J) \otimes \widetilde{H}^*(Y\slash P_{J'}) \ar[r]^{\quad\ \ \ \  \, \scalebox{0.85}{$\widetilde{\Cup}^{(\mathbb{D},\mathbb{S})}$}} & \widetilde{H}^*(Y\slash P_{J\cup J'}) 
                 }  
            \end{equation}  
            
         \begin{itemize}
   \item If $J\cap J'=\varnothing$ or $J\cap J' \neq \varnothing$ but $n_j=0$ for all $j\in J\cap J'$, then\\   
  $ \widetilde{H}^*(Y\slash P_J) \otimes \widetilde{H}^{*}(Y\slash P_{J'}) \xlongrightarrow{\widetilde{\Cup}^{(\mathbb{D},\mathbb{S})}}
    \widetilde{H}^{*}(Y\slash P_{J\cup J'})$ is the product $\widetilde{\cup}$
    induced from the relative cup product 
     $ H^*(Y,P_J) \otimes H^{*}(Y,P_{J'}) \xlongrightarrow{\cup}
    H^{*}(Y,P_{J\cup J'})$.\nn
          
     \item If $J\cap J' \neq \varnothing$ and there exists $n_j\geq 1$ for some $j\in J\cap J'$, then
    \\   
  $ \widetilde{H}^*(Y\slash P_J) \otimes \widetilde{H}^{*}(Y\slash P_{J'}) \xlongrightarrow{\widetilde{\Cup}^{(\mathbb{D},\mathbb{S})}}
    \widetilde{H}^{*}(Y\slash P_{J\cup J'})$ is trivial.
    \end{itemize}        
            
  It is clear that $\big( \mathcal{R}^*_{Y,\mathcal{P}},
   \Cup^{(\mathbb{D},\mathbb{S})} \big)$ and
   $\big( \widetilde{\mathcal{R}}^*_{Y,\mathcal{P}},
   \widetilde{\Cup}^{(\mathbb{D},\mathbb{S})} \big)$
   determine each other.\n

  \begin{thm} \label{Thm:Main-Dn-Sn-1}
   For any panel structure $\mathcal{P}=\{P_j\}_{j\in [m]}$ on a CW-complex $Y$, there is a ring isomorphism (up to a sign) from 
   $\big( \mathcal{R}^*_{Y,\mathcal{P}}, \Cup^{(\mathbb{D},\mathbb{S})} \big)$ to the integral 
   cohomology ring of  $(\mathbb{D},\mathbb{S})^{(Y,\mathcal{P})}$.  
    Moreover, we can make the ring isomorphism degree-preserving by shifting the degrees of all the elements in $H^*(Y,P_J)$ up by $\sum_{j\in J} n_j $ for every $J\subseteq [m]$. 
  \end{thm}     
  \begin{proof}
   The proof is parallel to the proof of~\cite[Theorem 4.8\,(b)]{Yu20}.           
    \end{proof}
    \n
    
    By Theorem~\ref{Thm:Main-Dn-Sn-1}, we have the following results
    for $(D^2,S^1)^{(Y,\mathcal{P})}$ and $(D^1,S^0)^{(Y,\mathcal{P})}$.
    
    \begin{cor} \label{Cor:D2-S1}   
   For any panel structure $\mathcal{P}=\{P_j\}_{j\in [m]}$ on a finite CW-complex $Y$, there is a ring isomorphism
   (up to a sign) from $\big( \mathcal{R}^*_{Y,\mathcal{P}}, \Cup \big)$ to the integral cohomology ring of $(D^2,S^1)^{(Y,\mathcal{P})}$, where the $\Cup$ on $\mathcal{R}^*_{Y,\mathcal{P}}$
   is defined by:
    \begin{itemize}
       \item If $J\cap J' \neq \varnothing$,
     $ H^*(Y,P_J) \otimes H^{*}(Y,P_{J'}) \overset{\Cup}{\longrightarrow}
    H^{*}(Y,P_{J\cup J'})$ is trivial.\n
    
     \item If $J\cap J'=\varnothing$, 
  $ H^*(Y,P_J) \otimes H^{*}(Y,P_{J'}) \overset{\Cup}{\longrightarrow}
    H^{*}(Y,P_{J\cup J'})$ is the relative cup product $\cup$.     
    \end{itemize} 
    Moreover, we can make this ring isomorphism
     degree-preserving by shifting the degrees of all the
     elements in $H^*(Y,P_J)$ up by $|J|$ for every $J\subseteq [m]$.
    \end{cor}
    
     \begin{cor} \label{Cor:D1-S0}   
   For any panel structure $\mathcal{P}=\{P_j\}_{j\in [m]}$ on a finite CW-complex $Y$, there is a graded ring isomorphism from $\big( \mathcal{R}^*_{Y,\mathcal{P}}, \cup \big)$ to the integral cohomology ring of $(D^1,S^0)^{(Y,\mathcal{P})}$, where the $\cup$ on $\mathcal{R}^*_{Y,\mathcal{P}}$ is the relative cup product
    $$ H^*(Y,P_J) \otimes H^{*}(Y,P_{J'}) \overset{\cup}{\longrightarrow}
    H^{*}(Y,P_{J\cup J'}), \ \forall J,J'\subseteq [m].$$ 
    \end{cor}
    
    Note that here we have an authentic graded ring isomorphism (without the sign factors) from $\big( \mathcal{R}^*_{Y,\mathcal{P}}, \cup \big)$ to 
    the integral cohomology ring of $(D^1,S^0)^{(Y,\mathcal{P})}$ (see the remark of~\cite[Corollary 4.10]{Yu20}).
    
    \vskip .2cm
  
  \subsection{Equivariant Cohomology of $(D^2,S^1)^{(Y,\mathcal{P})}$ and $(D^1,S^0)^{(Y,\mathcal{P})}$} \label{Sec:Equiv-Cohom}
  \ \n
    Let $Y$ be a finite CW-complex with a panel structure $\mathcal{P}=\{ P_j \}_{j\in [m]}$. Let
    $$D^2 =\{ z\in \mathbb{C}\,|\, \|z\|\leq 1\}, \ \
     S^1=\partial D^2.$$
     
      The natural action of $S^1$
   on $(D^2,S^1)$ induces a canonical action of $T^m=(S^1)^m$
   on the polyhedral product $(D^2,S^1)^{(Y,\mathcal{P})}$.\n 
    
   The equivariant cohomology of 
   $(D^2,S^1)^{(Y,\mathcal{P})}$, denoted by $H^*_{T^m}\big((D^2,S^1)^{(Y,\mathcal{P})}\big)$, is the cohomology of the \emph{Borel construction} 
     $$  ET^m \times_{T^m} (D^2,S^1)^{(Y,\mathcal{P})} = 
     ET^m \times (D^2,S^1)^{(Y,\mathcal{P})} \big\slash \sim $$
    where $(e,x)\sim (eg,g^{-1}x)$ for any $e\in 
    ET^m$, $x\in (D^2,S^1)^{(Y,\mathcal{P})}$ and $g\in T^m$.
    Here  
    $$ET^m = (ES^1)^m = (S^{\infty})^m.$$ \n
    
    Associated to the Borel construction, there is a canonical fiber bundle
    \begin{equation} \label{Equ:Borel-Fiber}
     (D^2,S^1)^{(Y,\mathcal{P})} \rightarrow ET^m \times_{T^m} (D^2,S^1)^{(Y,\mathcal{P})} \rightarrow BT^m \end{equation}
     where $BT^m = (BS^1)^m = (S^{\infty}\slash S^1)^m=(\mathbb{C}P^{\infty})^m$ is the \emph{classifying space} of $T^m$.\n
     
   By the colimit construction of $(D^2,S^1)^{(Y,\mathcal{P})}$ in~\eqref{Equ:Colimit-Construc}, the Borel construction
    \begin{align*}
  ET^m \times_{T^m} (D^2,S^1)^{(Y,\mathcal{P})} &=  
  \bigcup_{(f, L)\in \mathscr{P}_{(Y,\mathcal{P})}} 
 ET^m \times_{T^m}  (D^2,S^1)^{(f, L)} \\
  &=  \bigcup_{(f, L)\in \mathscr{P}_{(Y,\mathcal{P})}} 
   (S^{\infty}\times_{S^1} D^2, S^{\infty}\times_{S^1} 
   S^1)^{(f, L)} \\
   &= (S^{\infty}\times_{S^1} D^2, S^{\infty}\times_{S^1} 
   S^1)^{(Y,\mathcal{P})}.
\end{align*}
  
    Then by the homotopy equivalence of the pairs
    $$(S^{\infty}\times_{S^1} D^2, S^{\infty}\times_{S^1} 
   S^1) \rightarrow (\mathbb{C}P^{\infty}, * ),$$
   we can derive from~\cite[Corollary 4.5]{BBCG10} that
   there exists a homotopy equivalence
   $$  (S^{\infty}\times_{S^1} D^2, S^{\infty}\times_{S^1} 
   S^1)^{(Y,\mathcal{P})} \simeq (\mathbb{C}P^{\infty}, * )^{(Y,\mathcal{P})}.$$
   
   We call $(\mathbb{C}P^{\infty}, * )^{(Y,\mathcal{P})}$ the \emph{Davis-Januszkiewicz space} of $(Y,\mathcal{P})$, denoted by $\mathcal{DJ}(Y,\mathcal{P})$.
   So the equivariant cohomology ring of $(D^2,S^1)^{(Y,\mathcal{P})}$ with respect to the canonical $T^m$-action
   is isomorphic to the ordinary cohomology ring of $\mathcal{DJ}(Y,\mathcal{P})$. Similarly, we can prove that the equivariant cohomology ring of $(D^1,S^0)^{(Y,\mathcal{P})}$ with respect to the canonical $(\Z_2)^m$-action
   is isomorphic to the ordinary cohomology ring of 
   $(\mathbb{R}P^{\infty}, * )^{(Y,\mathcal{P})}$. \n
   
   To describe the cohomology ring structures 
   of $(\mathbb{C}P^{\infty}, * )^{(Y,\mathcal{P})}$ and 
   $(\mathbb{R}P^{\infty}, * )^{(Y,\mathcal{P})}$, we introduce
   the following definition which generalizes the classical notion of face ring (Stanley-Reisner ring) of a simplicial complex; see Lemma~\ref{Lem:Face-Ring} for an explanation.\n
   
     Let $\mathbf{k}$ denote a commutative ring with  unit.  For any $J\subseteq [m]$, let $R^J_{\mathbf{k}}$ be the subring of
   the polynomial ring $\mathbf{k}[x_1,\cdots, x_m]$ defined by
 \[ 
    R^J_{\mathbf{k}} : = 
      \begin{cases}
   \mathrm{span}_{\mathbf{k}}\{ x^{n_1}_{j_1}\cdots x^{n_s}_{j_s} \,|\,
       n_1>0, \cdots, n_s>0 \} ,  &  \text{if $J=\{ j_1,\cdots, j_s \}\neq \varnothing$}; \\
  \ \ \mathbf{k},  &  \text{if $J=\varnothing$}.
 \end{cases} 
  \]
  
    We can multiply $f(x)\in  R^J_{\mathbf{k}}$ and 
    $f'(x)\in R^{J'}_{\mathbf{k}}$ in 
    $\mathbf{k}[x_1,\cdots, x_m]$ and obtain an element
    $f(x)f'(x)\in R^{J\cup J'}_{\mathbf{k}}$.\n
    
  \begin{defi}[Topological Face Ring of a Space with Faces] \label{Def:Face-Ring-Panel-Struc}
    Suppose $Y$ is a CW-complex with a panel structure $\mathcal{P}=\{ P_j \}_{j\in [m]}$.  For any coefficients ring $\mathbf{k}$, the \emph{topological face ring of $(Y,\mathcal{P})$} over $\mathbf{k}$ is defined to be
     $$ \mathbf{k}(Y,\mathcal{P}) := \bigoplus_{J\subseteq [m]} H^*( P_{\cap J};\mathbf{k})\otimes R^J_{\mathbf{k}}. $$
   The product $\star$ on $\mathbf{k}(Y,\mathcal{P})$ is defined by:
     for any $J,J'\subseteq [m]$, 
     $$ \Big( H^*( P_{\cap J};\mathbf{k})\otimes R^J_{\mathbf{k}} \Big) \otimes
     \Big( H^*( P_{\cap J'};\mathbf{k})\otimes R^{J'}_{\mathbf{k}} \Big) \xlongrightarrow{\,\ \scalebox{0.9}{$\star$}\,\ }
     \Big( H^*( P_{\cap (J\cup J')};\mathbf{k})\otimes R^{J\cup J'}_{\mathbf{k}} \Big) $$
    for any $\phi\in H^*( P_{\cap J};\mathbf{k})$, 
     $\phi'\in H^*( P_{\cap J'};\mathbf{k})$ and
     $f(x)\in R^J_{\mathbf{k}}$, $f'(x)\in R^{J'}_{\mathbf{k}}$, 
    \begin{equation}\label{Equ:Face-Ring-Product}
      (\phi \otimes f(x))\star (\phi'\otimes f'(x)):=
      \big(\kappa^*_{J\cup J', J}(\phi)\cup
   \kappa^*_{J\cup J', J'}(\phi') \big) \otimes f(x)f'(x)
   \end{equation}
    where $\kappa_{I', I} : P_{\cap I'}\rightarrow P_{\cap I}$ is the inclusion map for any subsets $I\subseteq I'\subseteq [m]$ and $\kappa^*_{I', I}:
    H^*( P_{\cap I};\mathbf{k}) \rightarrow H^*( P_{\cap I'};\mathbf{k})$ is the induced homomorphism on 
    cohomology.
  \end{defi}
  
  In addition, we can consider $\mathbf{k}(Y,\mathcal{P})$ as a graded ring if we choose a degree for every indeterminate $x_j$
  in $\mathbf{k}[x_1,\cdots, x_m]$ and define
  $$\mathrm{deg}\big(\phi\otimes (x^{n_1}_{j_1}\cdots x^{n_s}_{j_s})\big) = \mathrm{deg}(\phi) + n_1\mathrm{deg}(x_{j_1})+\cdots + n_s\mathrm{deg}(x_{j_s}).$$

    We have the following theorem which generalizes~\cite[Theorem 1.7]{Yu20}.
    
  \begin{thm} \label{Thm:Equivariant-Cohomology-Y-F}
  Suppose $Y$ is a finite CW-complex and
  $\mathcal{P}=\{ P_j \}_{j\in [m]}$ is a panel structure on $Y$. Then the equivariant cohomology ring of $(D^2,S^1)^{(Y,\mathcal{P})}$ (or $(D^1,S^0)^{(Y,\mathcal{P})}$)  with $\Z$-coefficients (or $\Z_2$-coefficients) with respect to the canonical $(S^1)^m$-action  (or $(\Z_2)^m)$-action) is isomorphic as a graded ring to the topological face ring $\Z(Y,\mathcal{P})$ (or $\Z_2(Y,\mathcal{P})$) of $(Y,\mathcal{P})$ by choosing
    $\mathrm{deg}(x_j)=2$ (or $\mathrm{deg}(x_j)=1$) for all $1\leq j \leq m$.
  \end{thm}
  \begin{proof}
  The proof is parallel to the proof of~\cite[Theorem 1.7]{Yu20}. 
   \end{proof}

   From the canonical fiber bundle associated to the Borel construction in~\eqref{Equ:Borel-Fiber}, we 
   have a natural $H^*(BT^m)$-module structure on
   $H^*_{T^m}\big((D^2,S^1)^{(Y,\mathcal{P})}\big)$. By the identification
   $H^*(BT^m) = \Z[x_1,\cdots, x_m]$,
    we can write the $H^*(BT^m)$-module structure on $H^*_{T^m}\big((D^2,S^1)^{(Y,\mathcal{P})}\big)$
   as: for each $1\leq i \leq m$,
   \begin{equation} \label{Equ-BT-module-struc}
    x_i \cdot (\phi \otimes f(x)) = (1\otimes x_i)\star (\phi \otimes f(x)) \overset{\eqref{Equ:Face-Ring-Product}}{=} 
   \kappa^*_{J\cup\{i\}, J}(\phi) \otimes x_i f(x)
\end{equation}
   where $\phi\in  H^*(P_{\cap J})$ and
   $f(x)\in R^J_{\Z}$, $J\subseteq [m]$. 
  
   \vskip .4cm
   
 \section{Reexamining moment-angle complexes and
   polyhedral products over simplicial complexes} \label{Sec:Reexam-Simplical-Complex}
   
    In this section, we reexamine the constructions of  
   (generalized) moment-angle complexes and polyhedral products
   over simplicial complexes from
   the viewpoint of panel structures. We will
   reinterpret some known results using our theorems 
    in the last three sections and obtain a few new results. 
  \n
  
   Let $K$ be a simplicial complex 
      on $[m]=\{1,2,\cdots,m\}$.
      We assume that the set $[m]$ is \emph{minimal} for $K$ in the sense that
   every $j\in [m]$ belongs to at least one simplex of $K$. We also use $v_1,\cdots, v_m$ to refer to the vertices ($0$-dimensional simplex) of $K$. We do not distinguish $K$ and
   its geometrical realization below.\n
      \begin{itemize}
      \item For any $J\subseteq [m]$, let $K_J$ denote the \emph{full subcomplex} of $K$ consisting of all simplices of $K$ which have all of their vertices in $J$, i.e. 
   $$  K_J = \{ \sigma\cap J\,|\, \sigma \in K \}.$$
   
  \item Let $K'$ denote the barycentric subdivision of $K$. We can consider
    $K'$
   as the set of chains of simplices in $K$ ordered by inclusions.  For each nonempty simplex
   $\sigma\in K$, let $F_{\sigma}$ denote the geometrical realization of the poset
   $$K_{\geq \sigma}=\{ \tau\in K\,|\, \sigma\subseteq \tau \}.$$
    Thus,  
    $F_{\sigma}$
   is the subcomplex of $K'$ consisting of all simplices of the form
   $$\sigma=\sigma_0 \subsetneq \sigma_1 \subsetneq \cdots \subsetneq \sigma_l.$$ 
   It is clear that $F_{\sigma}$ is contractible.
 \n
 
 \item   Let $Y^K = u_0K'$ denote the cone on $K'$ over a point $u_0$ (see Figure~\ref{p:Subdivision-1}). Notice that $Y_K=F_{\varnothing}$.\n
         
     \item   There is a natural panel structure $\mathcal{P}$ on $Y^K$ defined by:
    \begin{equation*}
     \mathcal{P} = \{ P_j = F_{v_j}\}_{j\in [m]}.
   \end{equation*}
   \end{itemize}
   
   \begin{figure}
         \includegraphics[width=0.6\textwidth]{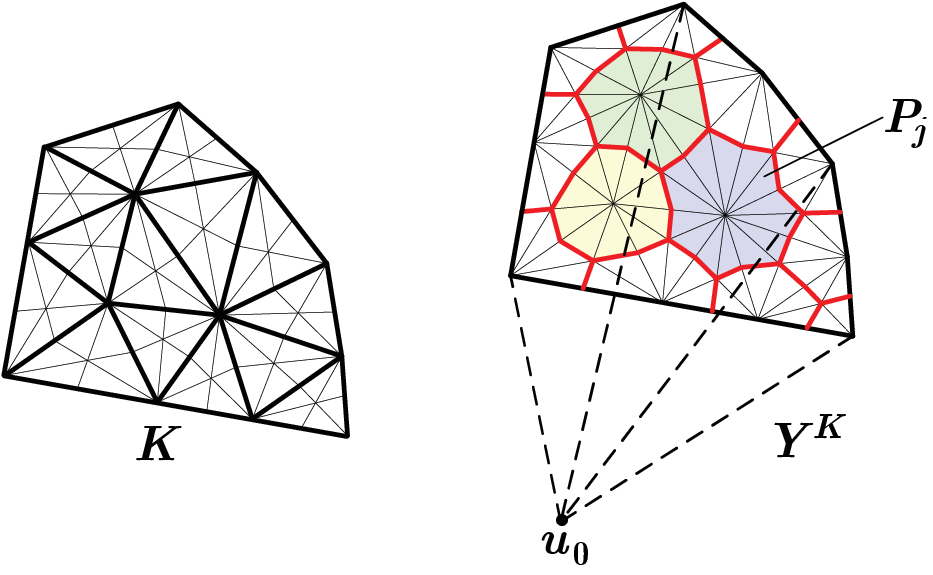}\\
          \caption{The cone of the barycentric subdivision of $K$}\label{p:Subdivision-1}
   \end{figure}
  
  This panel structure $\mathcal{P}$ on $Y^K$ has been studied in~\cite[p.314]{Da83} and~\cite[p.428]{DaJan91}.    
   Then by definition, for any $J\subseteq [m]$,
   \begin{itemize}
    \item If $J\neq \varnothing$, we have
    
    $ P_{\cap J} =\bigcap_{j\in J} F_{v_j}=   \begin{cases}
   F_{\sigma} ,  &  \text{if $J$ is the vertex set of a simplex $\sigma\in K$}; \\
   \varnothing ,  &  \text{otherwise}.
 \end{cases}$.\\
   If $J=\varnothing$, $P_{\cap \varnothing}=Y^K=F_{\varnothing}$. So all the faces of $(Y^K,\mathcal{P})$ are $\{ F_{\sigma}\,|\, \sigma\in K\}$.
       The space $Y^K$ together
    with its faces is called a \emph{simple polyhedral complex} in~\cite[p.428]{DaJan91}.\n
    
    \item For any $\sigma \in K$, the panel index of the face $F_{\sigma}$ in $(Y^K,\mathcal{P})$ is exactly $\sigma\subseteq [m]$, i.e. $I_{F_{\sigma}} =\sigma$. 
   \n
      
     \item $P_{J}$ is homotopy equivalent to the full subcomplex $K_J$ of $K$. This is because every $P_j$ and
     any of their intersections 
     $P_{\cap J}$ are contractible (if not empty), so
     $K_J$ can be thought of as the nerve complex of $P_{J}$ covered by $\{ P_j \}_{j\in J}$ (see~\cite[Corollary\,4G.3]{Hatcher02}). 
    Recall that the \emph{nerve complex} of a cover of a set $X$ 
    by a family of subsets $\{U_i\}_{1\leq i\leq n}$ is a 
    simplicial complex with $n$ vertices whose $(k-1)$-dimensional simplices correspond to the nonempty intersections of $k$ different $U_i$ in the cover. \n

  \end{itemize}
    
  \begin{lem} \label{Lem:Homotopy-1}
   For any $(\mathbb{X},\mathbb{A})= \{ (X_j,A_j,a_j) \}^m_{j=1}$, there is a homotopy equivalence 
   $$ (\mathbb{X},\mathbb{A})^{(Y^K,\mathcal{P})}
   \simeq (\mathbb{X},\mathbb{A})^{K}.$$
  \end{lem}  
    \begin{proof}
    Since the panel index of a face $F_{\sigma}$ in $(Y^K,\mathcal{P})$ is exactly $\sigma\subseteq [m]$, 
  the poset $\mathscr{P}_{(Y^K,\mathcal{P})}$ (see~\eqref{Equ:poset-Y-P}) determined by $(Y^K,\mathcal{P})$ is isomorphic to the following poset
 $$ \mathscr{P}_{K}:=\{ (\sigma, \tau)\,|\, \tau\subseteq \sigma \in K \}$$
 where $(\sigma,\tau)\leq (\sigma',\tau')$ if and only if $\sigma\subseteq \sigma'$ and $\sigma\backslash \tau \supseteq \sigma'\backslash \tau'$. Let 
 $$ 
  (\mathbb{X},\mathbb{A})^{(\sigma,\tau)} :=
   \prod_{j\in \sigma\backslash\tau} X_j \times \prod_{j\in [m]\backslash (\sigma\backslash\tau)} A_j, \ \ (\mathbb{X},\mathbb{A})^{(F_{\sigma},\tau)} :=
  F_{\sigma} \times \prod_{j\in \sigma\backslash\tau} X_j \times \prod_{j\in [m]\backslash (\sigma\backslash\tau)} A_j.$$
   
    Then we can think of both $(\mathbb{X},\mathbb{A})^{K}$ and $(\mathbb{X},\mathbb{A})^{(Y^K,\mathcal{P})}$ as colimits of diagrams over $\mathscr{P}_K$:
   $$ (\mathbb{X},\mathbb{A})^K = \bigcup_{\sigma\in K} \prod_{j\in\sigma} X_j \times \prod_{j\in [m]\backslash \sigma} A_j=
      \underset{(\sigma,\tau)\in \mathscr{P}_{K}}{\mathrm{colim}}\,  (\mathbb{X},\mathbb{A})^{(\sigma,\tau)},$$
      $$ (\mathbb{X},\mathbb{A})^{(Y^K,\mathcal{P})}=
      \bigcup_{\sigma\in K}  
       F_{\sigma} \times \prod_{j\in \sigma} X_j \times \prod_{j\in [m]\backslash \sigma} A_j =
      \underset{(\sigma,\tau)\in \mathscr{P}_{K}}{\mathrm{colim}} \, (\mathbb{X},\mathbb{A})^{(F_{\sigma},\tau)}. $$
  
 Moreover, there is an obvious map of diagrams over $\mathscr{P}_K$ defined by:
    $$  (\mathbb{X},\mathbb{A})^{(F_{\sigma},\tau)}  \longrightarrow
      (\mathbb{X},\mathbb{A})^{(\sigma,\tau)}, \ (\sigma,\tau)\in \mathscr{P}_K. $$
    Since $F_{\sigma}$ is contractible for all $\sigma\in K$ (including $\sigma=\varnothing$), the space $ (\mathbb{X},\mathbb{A})^{(Y^K,\mathcal{P})}$ is homotopy equivalent to $(\mathbb{X},\mathbb{A})^{K}$ by~\cite[Corollary 4.5]{BBCG10}. 
 \end{proof}
 
  For $(D^2,S^1)$ (or $(D^1,S^0)$), we can see that the homotopy equivalence in the above proof 
 from $(D^2,S^1)^{(Y^K,\mathcal{P})}$ (or $(D^1,S^0)^{(Y^K,\mathcal{P})}$) to $(D^2,S^1)^{K}$ (or $(D^1,S^0)^K$) is equivariant with respect to the canonical $(S^1)^m$-actions (or $(\Z_2)^m$-actions). \n

 So by Lemma~\ref{Lem:Homotopy-1}, we can compute the stable decomposition, cohomology ring and equivariant cohomology ring of $(\mathbb{X},\mathbb{A})^{K}$ in terms of 
 $(\mathbb{X},\mathbb{A})^{(Y^K,\mathcal{P})}$. 
  Next, we explain how to derive some known results of $(\mathbb{X},\mathbb{A})^{K}$ by applying our theorems to
  $(\mathbb{X},\mathbb{A})^{(Y^K,\mathcal{P})}$.

   \subsection{$(\mathbb{X},\mathbb{A})^{K}$
     where every $X_j$ in $(\mathbb{X},\mathbb{A})= \{ (X_j,A_j,a_j) \}^m_{j=1}$ is contractible} \label{Subsec:X_j-Contract}
     \ \n   
        
     We can derive a homotopy equivalence from Theorem~\ref{Thm:Main-General-X-Contract} for $(\mathbb{X},\mathbb{A})^{(Y^K,\mathcal{P})}$:
      \begin{equation*}  
         \mathbf{\Sigma}\big( (\mathbb{X},\mathbb{A})^{(Y^K,\mathcal{P})} \big) \simeq \bigvee_{J\subseteq [m]} 
          \mathbf{\Sigma}\Big(  ( Y^K\slash P_J ) \wedge  \bigwedge_{j\in J} A_j \Big).
      \end{equation*} 
      
      Since $P_J$ is homotopy equivalent to $K_J$, 
      \begin{equation} \label{Equ:Homotopy-Equiv-Yk-PJ}
         Y^K\slash P_J = u_0K\slash P_J \simeq u_0 P_J\slash P_J \simeq \mathbf{\Sigma}(P_J) \simeq \mathbf{\Sigma}(K_J).
       \end{equation}
      
   In the following if $J\subseteq [m]$ is the vertex set of a simplex in $K$, we simply write $J\in K$ by abuse of notation; otherwise we write $J\notin K$.
 \n
   If $J$ is the vertex set of a simplex $\sigma\in K$,
    then $K_J= \sigma$, which implies that $Y^K\slash P_J $ is contractible by~\eqref{Equ:Homotopy-Equiv-Yk-PJ}.
      Then since the smash product of a contractible space with any space is
      contractible, we obtain 
      \begin{align}  
   \ \ \     \mathbf{\Sigma}\big( (\mathbb{X},\mathbb{A})^{(Y^K,\mathcal{P})} \big) & \simeq \bigvee_{J\notin K, J\subseteq [m]} 
               \mathbf{\Sigma}\Big(  \mathbf{\Sigma}(K_J) \wedge  \bigwedge_{j\in J} A_j \Big)  \nonumber \\
   &\simeq \bigvee_{J\notin K, J\subseteq [m]} 
          \mathbf{\Sigma}(K_J) *  \bigwedge_{j\in J} A_j  \simeq \bigvee_{J\notin K, J\subseteq [m]} 
             \mathbf{\Sigma} \Big( K_J  * \bigwedge_{j\in J} A_j \Big) \label{Equ:Join-Smash}
      \end{align}       
     where
 $*$ denotes the join of two spaces. 
 The homotopy equivalences in~\eqref{Equ:Join-Smash} are due to the standard fact that for any based spaces
 $X$ and $Y$, 
  \[  \mathbf{\Sigma}(X)\wedge Y \simeq \mathbf{\Sigma}(X\wedge Y) \simeq X * Y. \]
 
 Note that the stable decomposition in~\eqref{Equ:Join-Smash}  recovers~\cite[Theorem\,2.21]{BBCG10}. 
 In particular, for 
 $(D^2,S^1)^K$ we recover the following result in~\cite[Corollary\,2.23]{BBCG10}:
    \begin{equation}  
      \mathbf{\Sigma}((D^2,S^1)^K) \simeq   \mathbf{\Sigma}\big( (D^2,S^1)^{(Y^K,\mathcal{P})} \big) \simeq \bigvee_{J\notin K, J\subseteq [m]} 
               \mathbf{\Sigma}^{|J|+2}(K_J). 
      \end{equation} 
  Moreover, for $(D^1,S^0)^K$, we obtain the following new result.
  \begin{prop}  \label{Prop:Stable-Decomp-RZ-K}
     $ \mathbf{\Sigma}((D^1,S^0)^K) \simeq   \mathbf{\Sigma}\big( (D^1,S^0)^{(Y^K,\mathcal{P})} \big) \simeq \bigvee_{J\notin K, J\subseteq [m]} 
               \mathbf{\Sigma}^{2}(K_J)$.
  \end{prop} 
   \vskip .2cm
        
       \subsection{$(\mathbb{X},\mathbb{A})^{K}$
     where every $A_j$ in $(\mathbb{X},\mathbb{A})= \{ (X_j,A_j,a_j) \}^m_{j=1}$ is contractible}
     \label{Subsec:A_j-Contract}
     \ \n     
     We can derive a homotopy equivalence from
      Theorem~\ref{Thm:Stable-Decomp-A-Contract}
   for $(\mathbb{X},\mathbb{A})^{(Y^K,\mathcal{P})}$:
   \[ \qquad \ S^1\vee \mathbf{\Sigma}\big( (\mathbb{X},\mathbb{A})^{(Y^K,\mathcal{P})} \big) \simeq \bigvee_{J\subseteq [m]} 
  \mathbf{\Sigma}\Big( \big( P_{\cap J}\cup y_0\big) \wedge  \bigwedge_{j\in J}  X_j \Big).\]
   \begin{itemize}
      \item If $J\neq \varnothing$ and $J\notin K$, then $P_{\cap J}$ is empty, which implies that the space
      $\big( P_{\cap J}\cup y_0\big) \wedge  \bigwedge_{j\in J}  X_j $ is contractible.\n
      
     \item If $J\neq \varnothing$ and $J\in K$ (suppose $J$ is the vertex set of $\sigma$), then $P_{\cap J} = F_{\sigma}$ is contractible, which implies that $\big( P_{\cap J}\cup y_0\big) \wedge  \bigwedge_{j\in J}  X_j \simeq \bigwedge_{j\in J}  X_j$.\n
      \item If $J= \varnothing$, then 
       $\big( P_{\cap J}\cup y_0\big) \wedge  \bigwedge_{j\in J}  X_j  = Y^K\cup y_0$.\nn
     \end{itemize}
     
     From the above discussion, we obtain 
      \[ \qquad \ S^1\vee \mathbf{\Sigma}\big( (\mathbb{X},\mathbb{A})^{(Y^K,\mathcal{P})} \big) \simeq  \mathbf{\Sigma}(Y^K\cup y_0) \vee 
      \bigvee_{\varnothing \neq J\in K} 
  \mathbf{\Sigma}\Big(\bigwedge_{j\in J}  X_j \Big).\]

  Then since $Y^K = u_0K'$ is contractible, 
  $\mathbf{\Sigma}(Y^K\cup y_0)\simeq S^1$ implies
   \begin{equation} \label{Equ:Homotopy-Equiv-A}
   \mathbf{\Sigma}\big( (\mathbb{X},\mathbb{A})^{(Y^K,\mathcal{P})} \big) \simeq   \bigvee_{\varnothing \neq J\in K} 
  \mathbf{\Sigma}\Big(\bigwedge_{j\in J}  X_j \Big). 
  \end{equation} 
     
     Note that the stable decomposition in~\eqref{Equ:Homotopy-Equiv-A} recovers~\cite[Theorem\,2.15]{BBCG10}.
            \vskip .2cm
              
     \subsection{Cohomology rings of $(D^2,S^1)^K$ and $(D^1,S^0)^K$} \label{Subsec:D2S1-Cohomology} \ \n
    
     The cohomology rings of $(D^2,S^1)^K$ and 
     $(D^2,S^1)^{(Y^K,\mathcal{P})}$ are isomorphic since they are homotopy equivalent. So by Corollary~\ref{Cor:D2-S1}, we obtain the following proposition.
     
      \begin{prop} \label{Prop:Moment-Angle} 
      There is a ring isomorphism (up to a sign) from
  $\big( \mathcal{R}^*_{Y^K,\mathcal{P}}, \Cup \big)$
  to the integral cohomology ring of $(D^2,S^1)^{K}$.
  Moreover, we can make this isomorphism degree-preserving by
  shifting the degrees of all the elements in $ H^*(Y^K,P_J)$ up by $|J|$ for every
  $J\subseteq [m]$.
  \end{prop}
  
  Since $Y^K$ is contractible, for any subset $J\subseteq [m]$,  
  we have an isomorphism
   $$ \Xi_J : H^q(Y^K,P_J)
   \xlongrightarrow{\cong} \widetilde{H}^{q-1}(P_J)  \xlongrightarrow{\cong} \widetilde{H}^{q-1}(K_J),\ \forall q\in \Z.$$ 
   So for any subsets $J,J'\subseteq [m]$, we have a commutative diagram
   \begin{equation}  \label{Equ:Diagram-Induced-Prod}
   \xymatrix{
           H^q(Y^K,P_J) \otimes H^{q'}(Y^K,P_{J'}) \ar[d]_{\cong} \ar[r]^{\quad\ \ \ \ \scalebox{0.85}{$\Cup$}}
                &   H^{q+q'}(Y^K,P_{J\cup J'}) \ar[d]^{\cong}  \\
           \widetilde{H}^{q-1}( P_J) \otimes \widetilde{H}^{q'-1}( P_{J'})\ar[d]_{\cong} \ar[r] & \widetilde{H}^{q+q'-1}(P_{J\cup J'}) \ar[d]^{\cong} \\
          \widetilde{H}^{q-1}( K_J) \otimes \widetilde{H}^{q'-1}( K_{J'}) \ar[r]^{\quad \ \ \  \scalebox{0.85}{$\sqcup$}} & \widetilde{H}^{q+q'-1}(K_{J\cup J'})      
                 }  
            \end{equation}  
      where $\sqcup$ is induced by $\Cup$ through the above isomorphisms $\Xi_J$. More precisely,      
      for any $\phi\in
       \widetilde{H}^*( K_J),\, \phi'\in \widetilde{H}^*( K_{J'})$, we have (see Corollary~\ref{Cor:D2-S1})
       \[ \phi\sqcup\phi' =  \begin{cases}
  \Xi_{J\cup J'} \big( \Xi^{-1}_J(\phi)\cup \Xi^{-1}_{J'}(\phi')\big),  &  \text{if $J\cap J'=\varnothing$}; \\
   \ 0  ,  &  \text{if $J\cap J'\neq \varnothing$}.
 \end{cases}   \]
     
    From the definitions of the isomorphism $\Xi_J$ and the relative cup product $\cup$, we can check that $\sqcup$ is equivalent to the cohomology product 
     of $(D^2,S^1)^K$ given in~\cite[Proposition 3.2.10]{BP15}.\n

       Similarly, by Corollary~\ref{Cor:D1-S0} we obtain 
       a parallel result for $(D^1,S^0)^{K}$
       which gives a new description of
       integral cohomology ring of $(D^1,S^0)^{K}$.
       
  \begin{prop} \label{Prop:Real-Moment-Angle}
  The integral
   cohomology ring of the real moment-angle complex $(D^1,S^0)^{K}$
  is isomorphic to $\big( \mathcal{R}^*_{Y^K,\mathcal{P}}, \cup \big)$ where
  $ \mathcal{R}^*_{Y^K,\mathcal{P}}=\bigoplus_{J \subseteq [m]} H^*(Y^K,P_J)$ and
  $H^*(Y^K,P_J) \otimes H^{*}(Y^K,P_{J'}) \overset{\cup}{\longrightarrow}
    H^{*}(Y^K,P_{J\cup J'})$ 
    is the relative cup product for all $J,J'\subseteq [m]$.\n
    \end{prop}
    
    \begin{rem}
     An earlier description of the integral cohomology ring of 
         $(D^1,S^0)^{K}$ was given by Cai~\cite{CaiLi17} in a very different way. 
   \end{rem} 
  
    \vskip .2cm
   
    \subsection{The equivariant cohomology rings of $(D^2,S^1)^{K}$ and $(D^1,S^0)^{K}$}  
     \label{Subsec:D2-S1-Equiv-Cohom} \ \n
    
   Since $(D^2,S^1)^{K}$ and $(D^2,S^1)^{(Y^K,\mathcal{P})}$ are equivariantly homotopy equivalent, 
   their equivariant cohomology rings are isomorphic.
  According to Theorem~\ref{Thm:Equivariant-Cohomology-Y-F}, the equivariant cohomology ring of $(D^2,S^1)^{K}$ is isomorphic to the topological face ring of $(Y^K,\mathcal{P})$:
      $$ \Z(Y^K,\mathcal{P}) = \bigoplus_{J\subseteq [m]} H^*( P_{\cap J} )\otimes R^J_{\Z},$$
  where the product $\star$ on $\Z(Y^K,\mathcal{P})$ is defined by (see~\eqref{Equ:Face-Ring-Product}): for any $J,J'\subseteq [m]$, 
     $$ \qquad \ \ \Big( H^*( P_{\cap J} )\otimes R^J_{\Z} \Big) \otimes
     \Big( H^*( P_{\cap J'})\otimes R^{J'}_{\Z} \Big) \xlongrightarrow{\,\ \scalebox{0.9}{$\star$}\,\ }
     \Big( H^*( P_{\cap (J\cup J')})\otimes R^{J\cup J'}_{\Z} \Big) $$   
     $$ \qquad \ (\phi \otimes f(x))\star (\phi'\otimes f'(x)):=
      \big(\kappa^*_{J\cup J', J}(\phi)\cup
   \kappa^*_{J\cup J', J'}(\phi') \big) \otimes f(x)f'(x)   $$
    where $\phi\in H^*( P_{\cap J})$, 
     $\phi'\in H^*( P_{\cap J'})$ and
     $f(x)\in R^J_{\Z}$, $f'(x)\in R^{J'}_{\Z}$.

   \begin{lem} $\Z(Y^K,\mathcal{P})$ is isomorphic to
    the face ring $\Z[K]$.
    \end{lem} \label{Lem:Face-Ring}
  \begin{proof} 
   For $J\subseteq [m]$, if $J$ is the vertex set of a simplex $\sigma$ in $K$, then $P_{\cap J} = F_{\sigma}$
    is contractible which implies $H^*( P_{\cap J} )\otimes R^J_{\Z}\cong R^J_{\Z}$. Otherwise $P_{\cap J}$
    is empty and $H^*( P_{\cap J} )\otimes R^J_{\Z}$ vanishes. So we can write
     \begin{equation} \label{Equ:Face-Ring-K}
       \Z(Y^K,\mathcal{P}) \cong \Big( \bigoplus_{\sigma\subseteq [m],\sigma\in K}  R^{\sigma}_{\Z}, \,\star \Big), 
       \end{equation}
     where for any $f(x)\in R^{\sigma}_{\Z}$, $f'(x)\in R^{\sigma'}_{\Z}$ with $\sigma,\sigma'\in K$,
  $$f(x)\star f'(x) =  \begin{cases}
   f(x)f'(x) ,  &  \text{if $\sigma\cup \sigma'\in K$}; \\
   \   0 ,  &  \text{otherwise}.
 \end{cases}$$
 
 By comparing~\eqref{Equ:Face-Ring-K} with the linear basis of $\Z[K]$ given in~\eqref{Equ:Face-Ring-Basis}, we can easily check that $\Z(Y^K,\mathcal{P}) $ is isomorphic to $\Z[K]$. 
 The lemma is proved.
 \end{proof}
 \n
  The above lemma implies that the integral equivariant cohomology ring of $(D^2,S^1)^K$
  is isomorphic to the face ring $\Z[K]$ of $K$.  
  Similarly, we can derive from Theorem~\ref{Thm:Equivariant-Cohomology-Y-F} that the equivariant $\Z_2$-cohomology
  ring of $(D^1,S^0)^K$ is 
  isomorphic to $\Z_2[K]$.
 So we recover the statement of~\cite[Theorem\,4.8]{DaJan91}.\n
  
  From the above discussion, we see that the topological face ring of
  a panel structure is a generalization of 
  the face ring of a simplicial complex from the equivariant cohomology viewpoint.\n

   \subsection{Polyhedral product over the panel structure on $Y^K$ determined by a partition of the vertex set of $K$}
 \ \n
 
  Let $\mathcal{J}=\{ J_1,\cdots, J_k\}$ be a partition of the vertex set $[m]$ of $K$. Then we can define a panel structure
  on $Y^K$ associated to $\mathcal{J}$ by:
    $$\mathcal{P}^{\mathcal{J}} := \big\{ P^{\mathcal{J}}_i = \bigcup_{j\in J_i} F_{v_j}\big\}^k_{i=1}.$$
 
  For any simplex $\sigma\in K$, let $V(\sigma)\subseteq [m]$ be the vertex set of $\sigma$.
   By a similar argument as for $(Y^K,\mathcal{P})$, 
   we can deduce that for any $L \subseteq [k]=\{1,\cdots, k\}$,
   \begin{itemize}
    \item If $L\neq \varnothing$, $P^{\mathcal{J}}_{\cap L}$ is the union of 
    all the $F_{\sigma}$ where $\sigma\in K$ and $V(\sigma)\cap J_i \neq\varnothing$ for all $i\in L$. If $L=\varnothing$,
    $P^{\mathcal{J}}_{\cap \varnothing} = Y^K$.
     \n
    
   \item $P^{\mathcal{J}}_{L}$ is homotopy equivalent to the full subcomplex $K_{J_L}$ of $K$ where 
   $$ J_L = \bigcup_{i\in L} J_i .$$
    
    \end{itemize}

  For any $(\mathbb{X},\mathbb{A})= \{ (X_i,A_i,a_i) \}^k_{i=1}$, we can apply our theorems to compute 
  the stable decomposition, cohomology ring 
   and equivariant cohomology ring of
   $(\mathbb{X},\mathbb{A})^{(Y^K,\mathcal{P}^{\mathcal{J}})}$ under the same conditions on
   $(\mathbb{X},\mathbb{A})$ in the previous subsections. The precise statements are left to the reader.\n
   
  For the pair $(D^2,S^1)$, we can think of $(D^2,S^1)^{(Y^K,\mathcal{P}^{\mathcal{J}})}$ as the quotient space of $(D^2,S^1)^{(Y^K,\mathcal{P})}$ by the canonical action of an
  $(m-k)$-dimensional subtorus of $(S^1)^m =\R^m\slash \Z^m$ determined by
   $$\qquad\quad \ \{ \varepsilon_j - \varepsilon_{j'} \,|\, j, j' \ \text{belong to the same}\ J_i \ \text{for some}\  1\leq i \leq k \}  \subseteq \Z^m.$$
  where $\varepsilon_j = \{0,\cdots, \overset{j}{1}, \cdots ,0\}\in \Z^m$.\n
    
   Similarly, for the pair $(D^1,S^0)$, we can think of $(D^1,S^0)^{(Y^K,\mathcal{P}^{\mathcal{J}})}$ as the quotient space of $(D^1,S^0)^{(Y^K,\mathcal{P})}$ by the canonical action of a subgroup of rank $m-k$ in $(\Z_2)^m$.
\n

 Using Theorem~\ref{Thm:Main-General-X-Contract}, we obtain the stable decomposition of $(D^2,S^1)^{(Y^K,\mathcal{P}^{\mathcal{J}})}$:
   \begin{equation*}  
    (D^2,S^1)^{(Y^K,\mathcal{P}^{\mathcal{J}})}  \simeq \bigvee_{L\subseteq [k]} 
               \mathbf{\Sigma}^{|J|+2}(K_{J_L}).
      \end{equation*} 
  
 Note that this formula recovers~\cite[Theorem\,1.3]{Yu19}.\n

 In addition, for any partition $\mathcal{J}=\{ J_1,\cdots, J_k\}$ of $[m]$ with $k=\dim(K)+1$,
  $(D^2,S^1)^{(Y^K,\mathcal{P}^{\mathcal{J}})}$ and $(D^1,S^0)^{(Y^K,\mathcal{P}^{\mathcal{J}})}$ can be considered as the generalization of
  the \emph{pull-back from the linear model} (see~\cite[Example 1.15]{DaJan91}) in the study of quasitoric manifolds and small covers.\n
  
  \vskip .4cm 
    
 \section{Reexamining moment-angle complexes and
   polyhedral products over simplicial posets}
    \label{Sec:Reexam-Simplicial-Posets}
    
       The notion of moment-angle complex $\mathcal{Z}_{\mathcal{S}}$ associated to a 
    simplicial poset $\mathcal{S}$ is introduced
    by L\"u-Panov~\cite{LuPanov11} as a generalization of the
    moment-angle complex of a simplicial complex.     
        A poset (partially ordered set) $\mathcal{S}$ with the order relation
    $\leq $ is called \emph{simplicial} if it has an initial element
    $\hat{0}$ and for each $\sigma \in \mathcal{S}$ the lower segment
    $[\hat{0},\sigma]=\{ \tau\in \mathcal{S}: \hat{0} \leq \tau \leq \sigma\}$
    is the face poset of a simplex.\n
    
      By the definition in~\cite{LuPanov11},  $\mathcal{Z}_{\mathcal{S}}$ is glued from the ``moment-angle blocks'' $(D^2,S^1)^{\sigma}$ according to the poset relation in $\mathcal{S}$. We can write $\mathcal{Z}_{\mathcal{S}}$
    as the colimit of a diagram on $\mathcal{S}$ defined by
    $\sigma\rightarrow(D^2,S^1)^{\sigma}$, that is
      $$\mathcal{Z}_{\mathcal{S}} = \underset{\sigma\in \mathcal{S}}{\mathrm{colim}}\,  (D^2,S^1)^{\sigma}.$$
    
   \n
        
    For each $\sigma \in \mathcal{S}$ we assign a geometrical simplex $\Delta^{\sigma}$
    whose face poset is $[\hat{0},\sigma]$. By gluing these geometrical simplices
    together according to the order relation in $\mathcal{S}$, we get a cell
    complex $\Delta^{\mathcal{S}}$ in which the closure of each cell is identified with a simplex
    preserving the face structure, and all attaching maps are inclusions.    
      We call $\Delta^{\mathcal{S}}$ the \emph{geometrical realization} of $\mathcal{S}$.
       For convenience, we still use $\Delta^{\sigma}$ to denote the image of each
    geometrical simplex $\Delta^{\sigma}$ in $\Delta^{\mathcal{S}}$. \n

    In the rest of this section, we assume that $\mathcal{S}$
    is a finite simplicial poset.
    Let the \emph{vertex set} of $\Delta^{\mathcal{S}}$ be
     $$V(\Delta^{\mathcal{S}})= [m]=\{ 1,\cdots, m\}.$$
     Besides, we also use $v_1,\cdots, v_m \in \mathcal{S}$ to denote the 
     elements corresponding to the vertices of $\Delta^{\mathcal{S}}$.\n
     
    For any $\sigma\in \mathcal{S}$, let
     $V(\Delta^\sigma)\subseteq [m]$ be the 
     \emph{vertex set} of the simplex $\Delta^{\sigma}$. Note that here two different simplices of $\Delta^{\mathcal{S}}$ may have the same vertex set.
    
    \begin{itemize}
      \item Let $\mathrm{Sd}(\Delta^{\mathcal{S}})$ denote the barycentric subdivision 
      of $\Delta^{\mathcal{S}}$.
     For each nonempty simplex
   $\Delta^{\sigma}\in \Delta^{\mathcal{S}}$, let $F_{\sigma}$ denote the geometrical realization of the poset $\mathcal{S}_{\geq \sigma} = \{ \tau \in \mathcal{S}\,|\, \sigma\leq \tau \}$.
    So $F_{\sigma}$ is the subcomplex of $\mathrm{Sd}( \Delta^{\mathcal{S}})$ consisting of all simplices of the form
   $\sigma=\sigma_0 \lneq \sigma_1 \lneq \cdots \lneq \sigma_l$.\n
    
    \item  For any $J\subseteq [m]$, let $\mathcal{S}_J$ denote the \emph{full subposet} of $\mathcal{S}$ consisting of all $\sigma\in \mathcal{S}$ 
    with $V(\Delta^{\sigma})\subseteq J$. Then
    the geometrical realization $\Delta^{\mathcal{S}_J}
    =\varnothing$ if and only if $\mathcal{S}_J=\hat{0}$.\n

   \item  Let $Y^{\mathcal{S}}=u_0\mathrm{Sd}(\Delta^{\mathcal{S}})$ denote the cone of 
   $\mathrm{Sd}(\Delta^{\mathcal{S}})$ over a point $u_0$. \n
    
   \item There is a natural panel structure $\mathcal{P}$ on $Y^{\mathcal{S}}$ defined by:
    \begin{equation*}
     \mathcal{P} = \{ P_j = F_{v_j}\}^m_{j=1}.
   \end{equation*}
  \n
  \end{itemize}
  
 For a subset $J\subseteq [m]$, 
 \begin{itemize}
   \item If $J\neq \varnothing$, $P_{\cap J}$ is the disjoint union of all $F_{\sigma}$ with $V(\Delta^{\sigma})= J$. Let 
      $$c_J = \text{the number of connected components of $P_{\cap J}$}.$$       
      Note that there is a one-to-one
   correspondence between the components
   $F_{\sigma}$ in $P_{\cap J}$ and the simplices $\Delta^{\sigma}\subseteq \Delta^{\mathcal{S}}$
   with $V(\Delta^{\sigma})= J$ (see Figure~\ref{p:Face-Intersection} for an example).
     If $J=\varnothing$, $P_{\cap \varnothing} = Y^{\mathcal{S}}$ is contractible.
    
      \begin{figure}
         \includegraphics[width=0.34\textwidth]{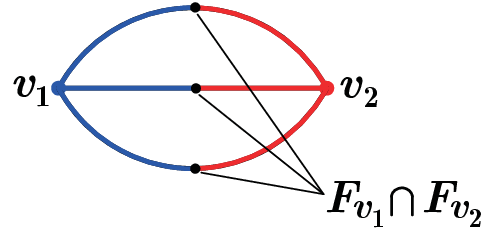}\\
          \caption{  }\label{p:Face-Intersection}
      \end{figure}
               
   \end{itemize}

  \begin{lem} \label{Lem:Poset-Homotopy-Equiv}
   For any $J\subseteq [m]$,  $P_J$ is homotopy equivalent to $\Delta^{\mathcal{S}_{J}}$.
  \end{lem}
  \begin{proof}
    The proof is analogous to the proof of~\cite[Corollary\,4G.3]{Hatcher02}. Let
  \[  \text{$\mathcal{U}_J =$  the cover of $P_J$ by 
    $\{ P_j = F_{v_j} \,|\, j\in J \}$.} \]
    
 Then $\mathcal{U}_J$ determines a \emph{diagram of spaces} $X_{\mathcal{U}_J}$ (see~\cite[Section\,4G]{Hatcher02} for the definition) where
  \begin{itemize}
   \item the vertices of $X_{\mathcal{U}_J}$ are the finite intersections of $P_j$'s in $\mathcal{U}_J$, and
   the edges of $X_{\mathcal{U}_J}$ are inclusions;\n
   \item the base of $X_{\mathcal{U}_J}$, denoted by 
  $\Gamma_J$, is the barycentric subdivision of the nerve complex $N(\mathcal{U}_J)$ of $\mathcal{U}_J$, . 
  \end{itemize}
  
  From the diagram $X_{\mathcal{U}_J}$, one can construct  
  a space $\Delta X_{\mathcal{U}_J}$ called the \emph{realization} of $X_{\mathcal{U}_J}$ (see~\cite[Section\,4G]{Hatcher02}). Moreover, 
  by~\cite[Proposition\,4G.2]{Hatcher02}, the space    $\Delta X_{\mathcal{U}_J}$ is homotopy equivalent to $P_J$. 
  \n
 For each $I\subseteq J$,  the
 intersection $\bigcap_{j\in I} P_j = P_{\cap I}$
 is the disjoint union of all $F_{\sigma}$'s with $V(\Delta^{\sigma})= I$, which has $c_I$ components.
 So we define another diagram of spaces
 over $\Gamma_J$, denoted by $Z_J$, where
 \begin{itemize}
  \item the vertices of $Z_J$ are
disjoint union of $c_I$ points;
\n
\item the edges of $Z_J$ consist of
maps sending a vertex corresponding to $F_{\sigma}$
to a vertex corresponding to $F_{\tau}$ with $\tau \subsetneq \sigma$.
\end{itemize}

Then by~\cite[Proposition\,4G.1]{Hatcher02}, the map $\Delta X_{\mathcal{U}_J} \rightarrow \Delta Z_J$ induced by sending each $F_{\sigma}$ to a point is a homotopy equivalence. On the other hand, since $\Gamma_J$
is the barycentric subdivision of the nerve complex of $\mathcal{U}_J$, it is routine to check that $\Delta Z_J$
is homeomorphic to $\Delta^{\mathcal{S}_{J}}$.
This proves the lemma.
      \end{proof}
   
 For any $(\mathbb{X},\mathbb{A})= \{ (X_j,A_j,a_j) \}^m_{j=1}$, we can construct the polyhedral product
 $(\mathbb{X},\mathbb{A})^{(Y^{\mathcal{S}},\mathcal{P})}$.  From our general theorems of polyhedral products over
 a manifold with faces, we can derive the following results.\n
    
      \begin{itemize}
      \item If every $X_j$ in $(\mathbb{X},\mathbb{A})= \{ (X_j,A_j,a_j) \}^m_{j=1}$ is contractible, we obtain from Theorem~\ref{Thm:Main-General-X-Contract} that
       \begin{equation*}  
         \mathbf{\Sigma}\big( (\mathbb{X},\mathbb{A})^{(Y^{\mathcal{S}},\mathcal{P})} \big) \simeq \bigvee_{J\subseteq [m]} 
          \mathbf{\Sigma}\Big(  ( Y^{\mathcal{S}}\slash P_J ) \wedge  \bigwedge_{j\in J} A_j \Big).
      \end{equation*} 
  Then since $Y^{\mathcal{S}}\slash P_J = u_0\mathrm{Sd}(\Delta^{\mathcal{S}})\slash P_J \simeq u_0 P_J\slash P_J \simeq \mathbf{\Sigma}(P_J) \simeq \mathbf{\Sigma}( \Delta^{\mathcal{S}_{J}})$,      
       \end{itemize}   
       \begin{align}  \label{Equ:X-A-S-Stable-Decomp}
       \mathbf{\Sigma}\big( (\mathbb{X},\mathbb{A})^{(Y^{\mathcal{S}},\mathcal{P})} \big)
   \simeq  \bigvee_{J\subseteq [m]} 
          \mathbf{\Sigma}\Big(  \mathbf{\Sigma}( \Delta^{\mathcal{S}_{J}}) \wedge  \bigwedge_{j\in J} A_j \Big) \simeq \bigvee_{J\subseteq [m]} 
             \mathbf{\Sigma} \Big( \Delta^{\mathcal{S}_J} * \bigwedge_{j\in J} A_j \Big).
      \end{align}    
    \begin{itemize} 
      \item If every $A_j$ in $(\mathbb{X},\mathbb{A})= \{ (X_j,A_j,a_j) \}^m_{j=1}$ is contractible,  Theorem~\ref{Thm:Stable-Decomp-A-Contract} gives
       \[ \qquad \ S^1\vee \mathbf{\Sigma}\big( (\mathbb{X},\mathbb{A})^{(Y^{\mathcal{S}},\mathcal{P})} \big) \simeq \bigvee_{J\subseteq [m]} 
  \mathbf{\Sigma}\Big( \big( P_{\cap J}\cup y_0\big) \wedge  \bigwedge_{j\in J}  X_j \Big).\]
  \begin{itemize}
    \item If $J\neq \varnothing$, 
     $P_{\cap J}$ is the disjoint union of all $F_{\sigma}$ with $V(\Delta^{\sigma})= J$. 
     Then since each $F_{\sigma}$ is contractible, $\big( P_{\cap J}\cup y_0\big) \wedge  \bigwedge_{j\in J}  X_j$ is homotopy equivalent to a wedge of $c_J$ 
     copies of $\bigwedge_{j\in J}  X_j$, denoted by
      $\bigvee_{c_J} \bigwedge_{j\in J}  X_j$.\nn
      
      \item If $J= \varnothing$, 
       $\big( P_{\cap J}\cup y_0\big) \wedge  \bigwedge_{j\in J}  X_j  = Y^{\mathcal{S}}\cup y_0$.\nn
     \end{itemize}
     
  From the above discussion, we obtain 
      \[ \qquad \ S^1\vee \mathbf{\Sigma}\big( (\mathbb{X},\mathbb{A})^{(Y^{\mathcal{S}},\mathcal{P})} \big) \simeq  \mathbf{\Sigma}(Y^{\mathcal{S}}\cup y_0) \vee \bigvee_{\varnothing\neq J\subseteq [m]} 
  \mathbf{\Sigma}\Big( \bigvee_{c_J} \bigwedge_{j\in J}  X_j \Big).\]
  Since $Y^{\mathcal{S}} = u_0\mathrm{Sd}(\Delta^{\mathcal{S}})$ is contractible, 
  $\mathbf{\Sigma}(Y^{\mathcal{S}}\cup y_0)\simeq S^1$.
  So we have
   \begin{equation} 
   \mathbf{\Sigma}\big( (\mathbb{X},\mathbb{A})^{(Y^{\mathcal{S}},\mathcal{P})} \big) \simeq   \bigvee_{\varnothing\neq J\subseteq [m]} 
  \mathbf{\Sigma}\Big( \bigvee_{c_J} \bigwedge_{j\in J}  X_j \Big). \end{equation}

 \end{itemize}
 \n
 
   In the rest of this section, we focus on the study of 
  $(D^2,S^1)^{(Y^{\mathcal{S}},\mathcal{P})}$.   
   There is a canonical $(S^1)^m$-action on $(D^2,S^1)^{(Y^{\mathcal{S}},\mathcal{P})}$. By an argument completely parallel to the proof of Lemma~\ref{Lem:Homotopy-1}, we can prove that
   $(D^2,S^1)^{(Y^{\mathcal{S}},\mathcal{P})}$ is equivariantly 
   homotopy equivalent to the moment-angle complex
   $\mathcal{Z}_{\mathcal{S}}$ over $\mathcal{S}$. Then by our Theorem~\ref{Thm:Main-General-X-Contract} (also see~\eqref{Equ:X-A-S-Stable-Decomp}), we obtain the following proposition.
   
   \begin{prop} \label{Prop:Simp-Poset-Stable-Decomp}
   For any finite simplicial poset $\mathcal{S}$ on $[m]$, 
    \begin{align}  \label{Equ:Z-S-Stable-Decomp}
\mathbf{\Sigma}\big( \mathcal{Z}_{\mathcal{S}}\big) \simeq \mathbf{\Sigma}\big( (D^2,S^1)^{(Y^{\mathcal{S}},\mathcal{P})} \big)
   \simeq  \bigvee_{J\subseteq [m]} 
          \mathbf{\Sigma}^{|J|+2}
          \big(\Delta^{\mathcal{S}_{J}}\big).
 \end{align}    
\end{prop}
   
   It was shown in~\cite[Theorem 3.5]{LuPanov11} that the integral cohomology ring of $\mathcal{Z}_{\mathcal{S}}$ is isomorphic as a graded algebra to
   the Tor algebra of the face ring $\Z[\mathcal{S}]$:
   \[ H^*(\mathcal{Z}_{\mathcal{S}};\Z)\cong \mathrm{Tor}_{\Z[v_1,\cdots,v_m]}(\Z[\mathcal{S}],\Z).  \]
   
   On the other hand, by our Corollary~\ref{Cor:D2-S1}, there is a ring isomorphism (up to a sign) between the integral cohomology ring of $(D^2,S^1)^{(Y^{\mathcal{S}},\mathcal{P})}$
 and
  $\big( \mathcal{R}^*_{Y^{\mathcal{S}},\mathcal{P}}, \Cup \big)$, where
  \[ \mathcal{R}^*_{Y^{\mathcal{S}},\mathcal{P}}=\bigoplus_{J \subseteq [m]} H^*(Y^{\mathcal{S}},P_J). \]  
  
  Moreover, we can make this isomorphism degree-preserving by shifting the degrees of all the elements in $H^*(Y^{\mathcal{S}}, P_J)$ up by $|J|$ for every $J\subseteq [m]$). Then since $\mathcal{Z}_{\mathcal{S}}$
  is homotopy equivalent to $(D^2,S^1)^{(Y^{\mathcal{S}},\mathcal{P})}$, we obtain the following proposition which gives
  a new topological description of 
  $\mathrm{Tor}_{\Z[v_1,\cdots,v_m]}(\Z[\mathcal{S}],\Z)$.\n
  
  \begin{prop} \label{Prop:Tor-Explain}
   For a finite simplicial poset $\mathcal{S}$ on $[m]$,
  there is a graded algebra isomorphism (up to a sign) between 
  $\mathrm{Tor}_{\Z[v_1,\cdots,v_m]}(\Z[\mathcal{S}],\Z)$
  and $\big( \mathcal{R}^*_{Y^{\mathcal{S}},\mathcal{P}}, \Cup \big)$ by shifting the degrees of all the elements in $H^*(Y^{\mathcal{S}}, P_J)$ up by $|J|$
  for every $J\subseteq [m]$.
  \end{prop}
  
  In addition, it was shown in~\cite[Corollary\,3.10]{LuPanov11} that $\mathrm{Tor}_{\Z[v_1,\cdots,v_m]}(\Z[\mathcal{S}],\Z)$ can be computed by a
  generalized Hochster's formula for the poset
  $\mathcal{S}$, from which one obtains the following formula:
   \begin{equation} \label{Equ:Z-S-Cohomology}
      H^p(\mathcal{Z}_{\mathcal{S}})\cong \bigoplus_{J \subseteq [m]} \widetilde{H}^{p-|J|-1}(\Delta^{\mathcal{S}_{J}}),\, \forall p\in \Z. 
  \end{equation}   
 Clearly, this formula also follows from our stable decomposition of $\mathcal{Z}_{\mathcal{S}}$ in~\eqref{Equ:Z-S-Stable-Decomp}.\n

   \n
   
 According to~\cite[Remark 3.2]{LuPanov11}, the equivariant
   cohomology ring of $\mathcal{Z}_{\mathcal{S}}$ with integral coefficients is isomorphic to the face ring $\Z[\mathcal{S}]$ of $\mathcal{S}$, that is:
   \begin{equation}  \label{Equ:Z-S-Equiv-Cohom}
      H^*_{T^m} (\mathcal{Z}_{\mathcal{S}}) = 
   H^*_{T^m}\big(\underset{\sigma\in \mathcal{S}}{\mathrm{colim}}\,  (D^2,S^1)^{\sigma}\big)\cong H^*\big(\underset{\sigma\in \mathcal{S}}{\mathrm{colim}}\,(\mathbb{C} P^{\infty}, p t)^{\sigma}\big) \cong \mathbb{Z}[\mathcal{S}]. 
   \end{equation}
    Recall that
   the \emph{face ring} of a simplicial poset $\mathcal{S}$ 
   with $\mathbf{k}$-coefficients is the quotient ring (see Stanley~\cite{St91})
        \begin{equation}\label{Equ:Face-Ring-Def-2}
           \mathbf{k}[\mathcal{S}]:= \mathbf{k}[\,\mathbf{v}_{\sigma} :\sigma\in \mathcal{S}\,] \slash 
          \mathcal{I}_{\mathcal{S}}  
         \end{equation}  
        where $\mathcal{I}_{\mathcal{S}}$ is the ideal generated by all the elements of 
        the form
        \[ \mathbf{v}_{\hat{0}}-1, \ \ \mathbf{v}_{\sigma}\mathbf{v}_{\tau} - \mathbf{v}_{\sigma\wedge \tau} \cdot 
                        \sum_{\eta\in \sigma\vee\tau} \mathbf{v}_{\eta}, \]
       where $\sigma\wedge\tau$ denotes the greatest common lower bound of $\sigma$ and $\tau$, 
       and $\sigma\vee\tau$ denotes the set of least
common upper bounds of $\sigma$ and $\tau$. The sum over the empty set is assumed to be zero. So we have $\mathbf{v}_{\sigma}\mathbf{v}_{\tau}=0$ in $\mathbf{k}[\mathcal{S}]$ if $\sigma\vee\tau=\varnothing$.   
      If $\mathcal{S}$
      is a simplicial complex $K$, 
       the definition~\eqref{Equ:Face-Ring-Def-2}
      of $\mathbf{k}[K]$ is related to the definition~\eqref{Equ:Face-Ring-Def-1}
      of $\mathbf{k}[K]$ by
      mapping $\mathbf{v}_{\sigma}$ to $\prod_{j\in \sigma} v_j$.\n

   On the other hand, by our Theorem~\ref{Thm:Equivariant-Cohomology-Y-F}, the equivariant cohomology ring of
   $(D^2,S^1)^{(Y^{\mathcal{S}},\mathcal{P})}$ with integral coefficients is isomorphic to the topological face ring
     \begin{equation*} 
      \Z(Y^{\mathcal{S}}, \mathcal{P}) := \bigoplus_{J\subseteq [m]} H^*( P_{\cap J} )\otimes R^J_{\Z}
      \end{equation*}
  where the product is defined in~\eqref{Equ:Face-Ring-Product}.  
  The following proposition implies that our result agrees with the conclusion~\eqref{Equ:Z-S-Equiv-Cohom} from~\cite{LuPanov11}.
 
 \begin{prop}
  There is a ring isomorphism from $\Z[\mathcal{S}]$
  to $\Z(Y^{\mathcal{S}}, \mathcal{P})$.
  \end{prop}
 \begin{proof}
   By our preceding discussion,
   $$H^n( P_{\cap J} )\cong    
    \begin{cases}
    \Z^{c_J} ,  &  \text{if $n=0$}; \\
    \, 0,  &  \text{otherwise};
 \end{cases} $$
  where
   $c_J$ is the number of connected components of 
   $P_{\cap J}$ which correspond to all the simplices
   $\Delta^{\sigma}\subseteq \Delta^{\mathcal{S}}$ with $V(\Delta^{\sigma})=J$.  For convenience, denote the vertex set of $\Delta^{\sigma}$ by $J_{\sigma}$, that is
     $$ J_{\sigma} = V(\Delta^{\sigma}) \subseteq [m].$$
   
   So we have
   \[ \Z(Y^{\mathcal{S}}, \mathcal{P}) \cong \bigoplus_{J\subseteq [m]}  
   \bigoplus_{\sigma\in\mathcal{S},J_{\sigma}=J} R^{J,\sigma}_{\Z} = \bigoplus_{\sigma\in\mathcal{S}} R^{J_{\sigma},\sigma}_{\Z}.  \]
   where $R^{J_{\sigma},\sigma}_{\Z} = R^{J_{\sigma}}_{\Z}$
   for any $\sigma\in \mathcal{S}$.  \n
   
   On the other hand, by~\cite[Corollary\,5.7]{MasPanov06}, every element of $\Z[\mathcal{S}]$ can be written uniquely as a linear combination
   \[  \sum_{\sigma_1\lneq \cdots \lneq \sigma_q} A(\sigma_1\lneq \cdots\lneq \sigma_q; l_1,\cdots, l_q)\cdot \mathbf{v}^{l_1}_{\sigma_1}\cdots \mathbf{v}^{l_q}_{\sigma_q} \]
   where $A(\sigma_1\lneq \cdots\lneq \sigma_q; l_1,\cdots,l_q)\in \Z$ and the sum is taken over all
   chains of elements $\sigma_1\lneq \cdots\lneq \sigma_q$
   in $\mathcal{S}$ with positive integers $l_1,\cdots, l_q$.\n
   
   Then according to the maximal element in the chain
   $\sigma_1\lneq \cdots \lneq \sigma_q$ in $\mathcal{S}$,
   we can write $\Z[\mathcal{S}]$ as a direct sum of subrings
   $$   \Z[\mathcal{S}] = \Z \oplus \bigoplus_{\hat{0}\neq \omega\in\mathcal{S}} \Z[\mathcal{S}]^{\omega}$$
   where $\Z[\mathcal{S}]^{\omega}$ consists of elements
   of the form
         \[  \sum_{\sigma_1\lneq \cdots \lneq \sigma_{q-1}\lneq\omega} A(\sigma_1\lneq \cdots\lneq \sigma_{q-1}\lneq \omega; l_1,\cdots, l_{q-1},l_q)\cdot\mathbf{v}^{l_1}_{\sigma_1}\cdots \mathbf{v}^{l_{q-1}}_{\sigma_{q-1}} \mathbf{v}^{l_q}_{\omega}. \]
   
  We can construct a ring homomorphism $\rho : 
     \Z[\mathcal{S}]\rightarrow  
     \Z(Y^{\mathcal{S}}, \mathcal{P})$ as follows.
     Define $\rho(\hat{0})=1$, and for any chain $\sigma_1\lneq \cdots\lneq \sigma_q$ in
   $\mathcal{S}$, let
     $$\rho(\mathbf{v}^{l_1}_{\sigma_1}\cdots \mathbf{v}^{l_q}_{\sigma_q}) = f_{\sigma_1}(x)^{l_1}\cdots f_{\sigma_q}(x)^{l_q} \in R^{J_{\sigma_q},\sigma_q}_{\Z}  $$
   where for any $\hat{0}\neq \sigma\in\mathcal{S}$, 
   $f_{\sigma}(x) = x_{i_1}\cdots x_{i_s}$ if $J_{\sigma}
   =\{v_{i_1},\cdots, v_{i_s}\}$. \n
   
  We claim that $\rho$ is a ring isomorphism.
   Indeed, for each $\hat{0}\neq \sigma\in\mathcal{S}$,
    we can think of 
   $R^{J_{\sigma},\sigma}_{\Z}$ as a subring of
   the face ring $\Z(\Delta^{\sigma})$ of $\Delta^{\sigma}$. So by the
   equivalence of the two definitions of $\Z(\Delta^{\sigma})$ given by~\eqref{Equ:Face-Ring-Def-1} and~\eqref{Equ:Face-Ring-Def-2}, the restriction of $\rho$ to $\Z[\mathcal{S}]^{\sigma}$ 
  is a ring isomorphism from $\Z[\mathcal{S}]^{\sigma}$ onto $R^{J_{\sigma},\sigma}_{\Z}$. Then the proposition follows.     
  \end{proof}
  
  By the above proposition, we can consider the topological face ring of
  a panel structure as a generalization of 
  the face ring of a simplicial poset as well. 
   
\vskip .4cm
    
  \section{Moment-angle complexes and
   polyhedral products over manifolds with corners}
    \label{Sec:Reexam-Manifold-Corners}
    
    An \emph{$n$-dimensional manifold with corners} $Q$ is a
   Hausdorff space together with a maximal
  atlas of local charts onto open subsets of $\R_{\geq 0}^n $
  such that the transitional functions are homeomorphisms which preserve the codimension of each point. Here
  the \emph{codimension} $c(x)$ of a point $x=(x_1,\cdots,x_n)$ in $\R_{\geq 0}^n$ is the number of
  $x_i$ which are $0$. So we have a well defined map
  $c: Q\rightarrow \Z_{\geq 0}$ where $c(q)$ is the codimension of a point $q\in Q$. In particular, the interior $Q^{\circ}$ of $Q$ consists of points of codimension $0$, i.e. $Q^{\circ} = c^{-1}(0)$. \n

   Suppose $Q$ is an $n$-dimensional manifold with corners with $\partial Q\neq \varnothing$.
  An \emph{open face} of $Q$ of codimension $k$
  is a connected component of $c^{-1}(k)$. A \emph{(closed) face}
  is the closure of an open face. A face of codimension $1$ is called
  a \emph{facet} of $Q$.
  Let $F_1,\cdots, F_m$ be all the facets of $Q$.\n 
  
      A cell decomposition of $Q$ is called \emph{compatible} with its manifold with corners structure  if all facets of $Q$ are subcomplexes in the cell decomposition. In the following, we assume that $Q$ is equipped with a compatible cell decomposition. Then there is a natural panel structure on $Q$ defined by
    \begin{equation}\label{Equ:Panel-Mfd-With-Corners}
     \mathcal{P} = \{ P_j = F_j\}^m_{j=1}. 
     \end{equation}
 
Then we can apply all our theorems proved in Section~\ref{Sec:Stable-Decomp}, \ref{Sec:Cohom-Ring} and~\ref{Sec:Equiv-Cohom} to obtain the corresponding results for any
  polyhedral product $(\mathbb{X},\mathbb{A})^{(Q,\mathcal{P})}$.  
  \n

  \begin{itemize}
    \item If $Q$ is a convex polytope in an Euclidean space,
      then $Q$ is a manifold with corners with an obvious compatible cell decomposition. The faces of $(Q,\mathcal{P})$ are exactly the faces of $Q$. Let 
      $K_Q$ be the nerve complex of $Q$ covered by
   its facets.       
    Since each face of $Q$ is also a convex polytope hence contractible, we can use a similar argument 
    as the proof of Lemma~\ref{Lem:Homotopy-1} to show that 
    $(\mathbb{X},\mathbb{A})^{(Q,\mathcal{P})}$ is homotopy
    equivalent to $(\mathbb{X},\mathbb{A})^{K_Q}$. In particular, we have $(D^2,S^1)^{(Q,\mathcal{P})}\simeq (D^2,S^1)^{K_Q}$. In addition, there is another notion called \emph{moment-angle space} of $Q$ (denoted by $\mathcal{Z}_Q$) defined in~\cite{AntonBuch11}.
    It was proved in~\cite{AntonBuch11} that 
    $\mathcal{Z}_Q$ is also homotopy equivalent to $(D^2,S^1)^{K_Q}$.
      \n
   
   \item The polyhedral product
   $(\mathbb{X},\mathbb{A})^{(Q,\mathcal{P})}$ generalizes 
   the polyhedral products over a nice manifold with corners defined in~\cite[Section 4]{Yu20}. Recall
   that a manifold with corners $Q$ is said to be \emph{nice} if either its boundary $\partial Q$ is empty or
  $\partial Q$ is non-empty and any codimension-$k$ face of $Q$ is a connected component of the intersection of
  $k$ different facets in $Q$. Moreover, the definition of $(\mathbb{X},\mathbb{A})^{(Q,\mathcal{P})}$ clearly makes sense for an arbitrary manifold with corners (not necessarily nice).
  \end{itemize}
  \vskip .5cm
   \section*{Acknowledgment}
   This work is partially supported by National Natural Science Foundation of China (grant no.11871266) and the PAPD (priority academic program development) of Jiangsu higher education institutions.
 \nn\nn
 \noindent \textbf{Competing Interests.} The author declares no competing interests pertaining to the
undertaken research.

\end{document}